\documentclass[11pt,a4paper,leqno]{article}

\usepackage{amssymb}
\usepackage{amsmath, amsthm}
\usepackage{epsfig}
\usepackage{geometry}
\def\RR{{\bf R}}
\def\barRR{\overline{\bf R}}
\def\ZZ{{\bf Z}}

\def\mGamma{\mathit{\Gamma}}

\numberwithin{equation}{section}
\newtheorem{Thm}{Theorem}[section]
\newtheorem{Prop}[Thm]{Proposition}
\newtheorem{Lem}[Thm]{Lemma}

\theoremstyle{definition}
\newtheorem{Clm}[Thm]{Claim}
\newtheorem{Def}[Thm]{Definition}
\newtheorem{Rem}[Thm]{Remark}
\newtheorem{Ex}[Thm]{Example}

\newenvironment{myitem}{
\refstepcounter{equation}
\\[1.0em] 
$(\thesection.\arabic{equation})$ \quad \quad
\begin{minipage}[t]{12.5cm}}
{\end{minipage}\\[1.0em]\noindent}

\newenvironment{myitem1}{
\refstepcounter{equation}
\\[1.0em]
$(\thesection.\arabic{equation})$  
\begin{minipage}[t]{13.5cm} \begin{itemize}}
{\end{itemize}\end{minipage}\\[1.0em]}

\newcommand{\opt}{\mathop{\rm opt} }
\newcommand{\dom}{\mathop{\rm dom} }

\title{L-extendable functions and a proximity scaling algorithm for minimum cost  multiflow problem}
\author{Hiroshi HIRAI \\
Department of Mathematical Informatics, \\
Graduate School of Information Science and Technology,   \\
University of Tokyo, Tokyo, 113-8656, Japan.\\
\texttt{\normalsize hirai@mist.i.u-tokyo.ac.jp}}

\begin{document}
\maketitle
\begin{abstract}
In this paper, we develop a theory of new classes of discrete convex functions, 
called L-extendable functions and alternating L-convex functions, defined on the product of trees. 
We establish basic properties for optimization: 
a local-to-global optimality criterion, the steepest descend algorithm by successive $k$-submodular function minimizations,  the persistency property, and the proximity theorem.
Our theory is motivated by minimum cost free multiflow problem.
To this problem, Goldberg and Karzanov gave two combinatorial weakly 
polynomial time algorithms based on capacity and cost scalings, 
without explicit running time.
As an application of our theory, 
we present a new simple polynomial proximity scaling algorithm 
to solve minimum cost free multiflow problem 
in $O( n  \log (n AC)  {\rm MF}(kn, km))$ time, 
where $n$ is the number of nodes, $m$ is the number of edges, $k$ is the number of terminals, 
$A$ is the maximum of edge-costs, $C$ is the total sum of edge-capacities, and 
${\rm MF}(n',m')$ denotes the time complexity to find a maximum flow in 
a network of $n'$ nodes and $m'$ edges.
Our algorithm is designed to solve,  in the same time complexity,  
a more general class of multiflow problems, 
minimum cost node-demand multiflow problem, and 
is the first combinatorial polynomial time algorithm to this class of problems. 
We also give an application to network design problem.
\end{abstract}
\section{Introduction}
An {\em L$^\natural$-convex function} (Favati-Tardella~\cite{FT90}, Murota~\cite{MurotaMPA}, Fujishige-Murota~\cite{FM00}) is 
a function $g$ on integer lattice $\ZZ^n$ satisfying so-called 
{\em discrete midpoint convexity inequality}:
\begin{equation}\label{eqn:midpoint_org}
g(x) + g(y) \geq g(\lfloor (x + y)/2\rfloor) + g(\lceil (x + y)/2\rceil)
\quad (x,y \in \ZZ^n),
\end{equation}
where $\lfloor \cdot \rfloor$ (resp. $\lceil \cdot \rceil$) 
is an operation on $\RR^n$ 
that rounds down (resp. up) the decimal fraction of each component. 
L$^\natural$-convex functions may be viewed as a $\ZZ^n$-generalization of
{\em submodular functions}, and
constitute a fundamental class of discrete convex functions 
in {\em discrete convex analysis} (Murota~\cite{MurotaBook}).
A representative example of L$^\natural$-convex function
 is a function $g$ represented as the following form:
\begin{equation}\label{eqn:typical}
g(x) = \sum_{i} g_i (x_i)  + \sum_{i,j} h_{ij}(x_i - x_j) \quad (x = (x_1,x_2,\ldots,x_n) \in \ZZ^n), 
\end{equation}
where $g_i$ and $h_{ij}$ are one-dimensional convex functions.
The minimization of such a function has both theoretical and practical interests; 
it is the dual of a minimum cost network flow problem, and 
has important applications in computer vision~\cite{KS09}.
Thus theory of L$^\natural$-convex functions provides
a unified treatment for optimizing these important classes of functions.

Let us mention some of particular features of L$^\natural$-convex functions $g$.
(1)~An optimality criterion of a local-to-global type~\cite[{Theorem 7.14}]{MurotaBook}:
For each point $x \in \ZZ^n$, 
a local minimization problem around $x$ is defined, and 
if $x$ is local optimal, then $x$ is global optimal.
Moreover this local minimization problem is a {\em submodular function minimization}, 
and is solvable in polynomial time~\cite{GLS, IFF, Schrijver}.
In particular,   if $x$ is not local optimal, 
then we can find another point $y$ with $g(y) < g(x)$; 
we naturally obtain $y$ with smallest $g(y)$.
The resulting descent algorithm, 
called the {\em steepest descent algorithm}, 
correctly outputs a global minimizer of $g$~\cite[Section 10.3.1]{MurotaBook}. 
The number of the descent steps is the $l_{\infty}$-distance
between the initial point and global minimizers~\cite{KS09, MurotaShioura14}.
(2) Proximity theorem 
(Iwata-Shigeno~\cite{IS02}, see \cite[Theorem 7.18]{MurotaBook}): 
For a minimizer $x$ over the set $(2\ZZ)^n$ of all even integral vectors, 
there is a global minimizer $y$ in the $l_{\infty}$-ball around $x$ with radius $n$.
This intriguing property is 
the basis of the proximity scaling algorithm for 
L$^\natural$-convex functions~\cite[Section 10.3.2]{MurotaBook}. 

Recently the L-convexity is considered for functions on 
general graph structures other than~$\ZZ^n$.
Observe that $\ZZ$ is naturally identified with the vertex set 
of a directed path (of infinite length), and
a function on $\ZZ^n$ is regarded as a function on the $n$-product of these paths.
Observe that operations $\lceil, \rceil$ and $\lfloor, \rfloor$ 
are definable in a graph-theoretical way.
Hence L$^\natural$-convex functions are well-defined functions on 
the Cartesian product of directed paths. 
Based on this observation, 
Kolmogorov~\cite{Kolmogorov11MFCS} 
considered an analogue of L$^{\natural}$-convex functions defined on
the product of rooted trees,  
called {\em tree-submodular functions}. 
Hirai~\cite{HH13SODA, HH_HJ13} considered 
an analogue of L$^{\natural}$-convex functions on 
a more general structure, a {\em modular complex}, 
which is a structure obtained by gluing of modular lattices.
His motivation comes from the tractability classification of minimum 0-extension problems
and a combinatorial duality theory of multicommodity flows.

In this paper, we continue this line of research.
We introduce the notion of {\em L-extendability} 
for functions on (the vertex set of) the $n$-fold Cartesian product $T^n$ of trees $T$.
This notion is inspired by the idea of 
{\em submodular relaxation}~\cite{GK13, IWY14, Kolmogorov12DAM} 
and related half-integral relaxations of NP-hard problems, 
such as vertex cover and multiway cut.
We first introduce a variation of a tree-submodular function,
called an {\em alternating L-convex function}. 
Alternating L-convex functions are also defined by a variation 
of the discrete midpoint convexity inequality (\ref{eqn:midpoint}), 
and coincide with Fujishige's {\em UJ-convex functions}~\cite{Fujishige14} 
if $T$ is a path and is identified with $\ZZ$.
As an analogue of half-integral integer lattice $(\ZZ/2)^n$, 
we consider the product $(T^*)^{n}$ of 
the {\em edge-subdivision} $T^*$ of $T$.
Then an {\em L-extendable function} is defined as 
a function $g$ on $T^n$
such that there is an alternating L-convex function $\bar g$
on $(T^*)^n$ such that its restriction to $T^n$ is equal to $g$, 
where $\bar g$ is called an {\em L-convex relaxation} of $g$.

The first half of our main contribution is 
to establish basic properties of alternating L-convex functions 
and L-extendable functions.
We show that alternating L-convex functions admit 
an optimality criterion of a local-to-global type.
Here the local problem is 
the problem of minimizing a {\em $k$-submodular function}, 
a generalization of submodular and bisubmodular functions 
introduced by Huber and Kolmogorov~\cite{HK12}.
This optimality criterion is an immediate consequence of the definition, and
may not be precisely new; it is expected from \cite{HH13SODA, Kolmogorov11MFCS}.
We further prove the $l_{\infty}$-geodesic property for 
the steepest descent algorithm for alternating L-convex functions:
The number of the iterations is equal to 
the $l_{\infty}$-distance from the initial point to minimizers.
We establish the proximity theorem 
for L-extendable functions:
Regard $T$ as a bipartite graph with two color classes $B,W$.
For a minimizer $x$ over $B^n$, 
there is a global minimizer of $g$ within the $l_{\infty}$-ball around $x$ with radius $n$.
We prove the persistency property for L-convex relaxations: 
a minimizer of an L-extendable function is obtained 
by {\em rounding} any minimizer of its L-convex relaxation. 
This property is known for the cases of 
bisubmodular and $k$-submodular relaxations~\cite{GK13, IWY14, Kolmogorov12DAM}. 
We introduce a useful special class of L-extendable functions, 
called {\em 2-separable convex functions}, 
as an analogue of a class of functions with form (\ref{eqn:midpoint}).
We give explicit L-convex relaxations for which the steepest descent algorithm
is implementable by a maximum flow algorithm.
In fact, the local problem is a minimization of a special $k$-submodular function, 
which is a sum of (binary) {\em basic $k$-submodular functions},  
introduced by Iwata, Wahlstr\"om, and Yoshida~\cite{IWY14}.
They showed that 
this class of $k$-submodular functions 
can be minimized by maximum flow computation.
Therefore the L-convex relaxation of a $2$-separable convex function 
is efficiently minimized.
For some cases, an optimal solution of this relaxation can easily be rounded to 
a $2$-approximate solution of the original 2-separable convex function. 
This approximation algorithm can be viewed as a generalization of 
the classical $2$-approximation algorithm for multiway cut~\cite{DJPSY94}.

These results are motivated by the design of a new simple polynomial scaling algorithm 
for {\em minimum cost free multiflow problem}, which is
the second half of our main contribution.
For an undirected (integer-)capacitated network with terminal set $S$,
a {\em multiflow} is a pair of 
a set of paths connecting terminals in $S$ and 
its flow-value function satisfying capacity constraints.
A {\em maximum free multiflow} is a multiflow of a maximum total flow-value.
Suppose that each edge has a nonnegative cost. 
The minimum cost free multiflow problem asks 
to find a maximum free multiflow with the minimum total cost.
Karzanov~\cite{Kar79} proved that there always exists a half-integral 
minimum cost maximum free multiflow, and presented a pseudo-polynomial 
time algorithm to find it.
Later he~\cite{Kar94} gave a strongly polynomial time algorithm 
by using a generic polynomial time LP solver 
(the ellipsoid method or the interior point method).
Currently no {\em purely combinatorial} strongly polynomial time algorithm is known.
Goldberg and Karzanov~\cite{GK97} presented two purely combinatorial 
weakly polynomial time algorithms: one of them is based on capacity scaling 
and the other one is based on cost scaling.
However the description and analysis of 
their algorithms (in each scaling phase) are not easy; 
they did not give an explicit polynomial running time.

As an application of the theory of L-extendable functions, 
we present a new simple purely 
combinatorial polynomial time scaling algorithm.
We hope that our algorithm will be a step toward 
the design of a purely combinatorial strongly polynomial time algorithm
for minimum cost multiflow problem.
We formulate a dual of our problem to 
the minimization of a 2-separable L-convex function on the product $T^n$ 
of a subdivided-star $T$, where a {\em subdivided star} 
is a tree obtained from a star by successive edge-subdivisions.
Then we can apply a domain scaling technique.
The scaled problem in each phase 
is again a minimization of a 2-separable L-convex function, and is solved by the steepest descent algorithm implemented by max-flow computations. 
The number of iterations is estimated by the proximity theorem 
the $l_{\infty}$-geodesic property, and the persistency property.
In the last phase, we obtain an optimal dual solution, 
and from this we can construct a desired minimum cost free multiflow.
Our algorithm may be viewed as a multiflow version of
a proximity scaling (or domain scaling) algorithm
for the convex dual of minimum cost flow problem~\cite{AHO04, KS09}, 
and is the first combinatorial algorithm to this problem 
having an explicit polynomial running time.
The total time is $O( n \log (n AC)  {\rm MF}(k n, km))$, where
$n$ is the number of nodes, $m$ is the number of edges, $k$ is the number of terminals,
$A$ is the maximum of an edge-cost, $C$ is the total sum of edge-capacities, 
and ${\rm MF}(n',m')$ is the time complexity of finding a maximum flow
in a network of $n'$ nodes and $m'$ edges.
Our algorithm is designed to solve, in the same time complexity,
 a more general class of multiflow problems, 
{\em minimum cost node-demand multiflow problems}, 
and is the first combinatorial polynomial time algorithm for this class of the problems. 
This multiflow problem arises as 
an LP-relaxation of a class of network design problems, 
called {\em terminal backup problems}~\cite{AK11, BKM13, XAC08INFOCOM}.
Recently Fukunaga~\cite{Fukunaga14}~gave a 4/3-approximation algorithm 
for capacitated terminal backup problem, based on rounding an LP solution 
(obtained by a generic LP-solver). 
Our algorithm gives a practical and combinatorial implementation of his algorithm.

We present results for L-extendable functions 
in Section~\ref{sec:L-extendable} and
results for minimum cost multiflow problem 
in Section~\ref{sec:multiflow}.

\paragraph{Notation.} 
Let $\RR$, $\RR_+$, $\ZZ$, and $\ZZ_+$ 
denote the sets of reals, nonnegative reals, integers, and nonnegative integers, respectively.
Let $\overline{\RR} := \RR \cup \{\infty\}$, 
where $\infty$ is an infinity element and is treated as: 
$x < \infty$ $(x \in \RR)$ and $\infty + x = \infty$ $(x \in \overline{\RR})$.

For a function $f: E \to \barRR$ on a set $E$, 
let $\dom f$ denote the set of elements $x \in E$ with $f(x) \neq \infty$.
For a subset $X \subseteq E$, let $f(X)$ denote $\sum_{x \in X} f(x)$.

For an undirected (directed) graph $G = (V,E)$, 
an edge between nodes $i$ and $j$ (from $i$ to $j$) is denoted by $ij$.
For a node subset $X$, let $\delta X$ denote 
the set of edges $ij$ with $i \in X$ and $j \not \in X$.
For nodes $s,t$, 
an  {\em $(s,t)$-cut} is a node subset $X$ with $s \in X \not \ni t$.
For a node subset $A$, an $(s,A)$-cut is a node subset $X$ with $s \in X \subseteq V \setminus A$.

\section{L-extendable functions}\label{sec:L-extendable}

In this section, we introduce alternating L-convex functions 
and L-extendable functions on the product of trees.
In Section~\ref{subsec:k-submodular}, 
we give preliminary arguments on $k$-submodular functions. 
In Section~\ref{subsec:alternating}, 
we introduce alternating L-convex functions, 
and establish the L-optimality criterion (Theorem~\ref{thm:L-optimality}) and 
the $l_\infty$-geodesic property (Theorem~\ref{thm:bound}) 
of the steepest descent algorithm.
In Section~\ref{subsec:L-extendable}, 
we introduce L-extendable functions, 
and establish the proximity theorem (Theorem~\ref{thm:proximity}) 
and the persistency property (Theorem~\ref{thm:persistency}).
In Section~\ref{subsec:2-separable}, 
we introduce 2-separable convex functions.
We show that the steepest descent algorithm for their L-convex relaxations 
is implementable by minimum cut computations 
(Theorems~\ref{thm:2separable} and \ref{thm:approx}).
Less obvious theorems are proved in Section~\ref{subsec:proof}.

\subsection{Preliminaries on $\pmb k$-submodular function}\label{subsec:k-submodular}
For a nonnegative integer $k$, 
let $S_k$ be a $(k+1)$-element set with specified element $0$.
Define a partial order $\preceq$ on $S_k$ 
by $0 \prec u$ for $u \in S_{k} \setminus \{0\}$ with no other relations.
Then $S_k$ is a meet-semilattice 
of minimum element $0$; in particular meet $\wedge$ exists.
For $u,v \in S_{k}$, 
define $u \sqcup v$ by $u \sqcup v := u \vee v$ if $u,v$ are comparable, 
and $u \sqcup v := 0$ otherwise.
For an $n$-tuple $\pmb{k} = (k_1, k_2,\ldots, k_n)$ of nonnegative integers,  
let $S_{\pmb k}$ denote the direct product $S_{k_1} \times S_{k_2} \times \cdots \times S_{k_n}$ of posets $S_{k_i}$ $(i=1,2,\ldots,n)$.
For $x = (x_1,x_2,\ldots,x_n), y = (y_1,y_2,\ldots,y_n) \in S_{\pmb k}$, 
let $x \wedge y := (x_1 \wedge y_1, x_2 \wedge y_2, \ldots, x_n \wedge y_n)$ 
and $x \sqcup y := (x_1 \sqcup y_1, x_2 \sqcup y_2, \ldots, x_n \sqcup y_n)$. 

\begin{Def}[\cite{HK12}]
$f: S_{\pmb  k} \to \barRR$ is {\em $\pmb k$-submodular} 
if it satisfies
\[
f(x) + f(y) \geq f(x \wedge y) + f(x \sqcup y) \quad (x,y \in S_{\pmb k}).
\]
\end{Def}
The notion of $\pmb k$-submodularity is introduced by Huber-Kolmogorov~\cite{HK12}. 
If $\pmb{k} = (k,k, \ldots,k)$,  
then $\pmb{k}$-submodular functions are particularly called $k$-submodular.
Then $1$-submodular functions are submodular functions, and $2$-submodular functions are bisubmodular functions.
Although both classes of functions can be minimized in polynomial time 
(under the oracle model)~\cite{FI05, GLS, IFF, MF10, Qi88, Schrijver}, 
it is not known 
whether general $\pmb k$-submodular functions 
can be minimized in polynomial time.
Recently, Thapper and \v{Z}ivn\'y~\cite{TZ12FOCS} 
discovered a powerful criterion for solvability of {\em valued CSP}, that is,  
a minimization of a function represented as a sum of functions 
with constant arity ($=$ the number of variables). 
As a consequence of their criterion, 
if $\pmb k$-submodular function $f$ is represented and given as
\[
f(x) = \sum_{i}  f_{i} (x_{i_1}, x_{i_2},\ldots,x_{i_m}) \quad (x = (x_1,x_2,\ldots,x_n) \in S_{\pmb k}),
\]
where each $f_{i}$ is $\pmb k$-submodular and the number $m$ of variables is constant, then 
$f$ can be minimized by solving a certain polynomial size linear program (BLP); 
see also~\cite{KTZ13}. 

We will deal with a further special class of 
$\pmb  k$-submodular functions, considered by~\cite{IWY14}, 
in which this class of $\pmb  k$-submodular functions can be efficiently minimized 
by a maximum flow algorithm.
Let $k,k'$ be nonnegative integers. 
For $a \in S_k$, 
let $\epsilon_a$ and $\theta_a$ be one-dimensional $k$-submodular functions on $S_k$ 
defined by
\begin{equation*}
\epsilon_a (u) :=
\left\{ 
\begin{array}{ll}
1 & {\rm if}\ u = a \neq 0, \\
0 & {\rm otherwise},
\end{array} \right. \\
\quad 
\theta_a(u) := \left\{ 
\begin{array}{ll}
-1 & {\rm if}\ u = a \neq 0, \\
0 & {\rm if}\  u= 0,\\
1 & {\rm otherwise}.
\end{array} \right.
\quad (u \in S_{k}).
\end{equation*}
It is not difficult to see that 
any one-dimensional $k$-submodular function $f$
is a nonnegative sum of $\epsilon_{a}$ and 
$\theta_a$ plus a constant:
\begin{equation}\label{eqn:unary}
f = f(0) + (f(0) - f(a)) \theta_a + \sum_{b \in S_k \setminus \{ 0, a\}} (f(b) - 2 f(0) + f(a))\epsilon_b,
\end{equation}
where $a$ is a minimizer of $f$ over $S_k$.
For $(a,a') \in S_{k} \times S_{k'}$, 
let $\mu_{a,a'}$ be  a function  on $S_k \times S_{k'}$ defined by
\begin{equation*}
\mu_{a,a'}(u,v) := \left\{
\begin{array}{ll}
0   & {\rm if} \ \mbox{$u = a \neq 0$ or $v = a' \neq 0$ or $u = v = 0$}, \\
1  & {\rm if}\ \mbox{$v = 0 \neq u \neq a$ or 
 $u= 0 \neq v \neq a'$},\\
2   & {\rm otherwise},
\end{array}\right. \quad ((u,v) \in S_k \times S_{k'}).
\end{equation*}

For a poset isomorphism $\sigma: S_k  \to S_{k'}$ 
(a bijection from $S_{k}$ to $S_{k'}$ with $\sigma(0) = 0$), 
let $\delta_{\sigma}$ be a function on $S_k \times S_{k'}$ defined by
\begin{equation*}
\delta_{\sigma}(u,v) := \left\{
\begin{array}{ll}
0  &{\rm if} \ v = \sigma (u), \\
1  & {\rm if}\ |\{u,v\} \cap \{0\}| = 1,   \\
2   & {\rm otherwise},
\end{array}\right. \quad ((u,v) \in S_k \times S_{k'}).
\end{equation*}
If $\sigma$ is the identity map ${\rm id}$, then $\delta_{\rm id}$ 
is denoted by $\delta$.

Observe that both $\mu_{a,a'}$ and $\delta_{\sigma}$ are $\pmb k$-submodular.
A (binary) {\em basic $\pmb k$-submodular function} $f$ on $S_{\pmb k}$ 
for $\pmb{k}= (k_1,k_2,\ldots, k_n)$
is a function represented as
\begin{description}
\item[Type I:] $f(x) =  f' (x_i)$  for some $i$ and one-dimensional $k$-submodular function $f' :S_{k_i} \to \overline\RR$,
\item[Type II:] $f(x) = \delta_{\sigma}(x_i,x_j)$
for distinct $i,j$ and an isomorphism $\sigma: S_{k_i} \to S_{k_j}$, or
\item[Type III:] $f(x) = \mu_{a,a'}(x_i, x_j)$ for distinct $i,j$ and $(a,a') \in S_{k_i} \times S_{k_j}$.
\end{description}
Iwata, Wahlstr\"om, and Yoshida~\cite{IWY14} showed
that a sum of basic $k$-submodular function 
can be efficiently minimized by any maximum flow algorithm.
\begin{Thm}[{\cite[Section 6]{IWY14}}]\label{thm:IWY14}
A nonnegative sum of $m$ basic $\pmb{k}$-submodular functions for $\pmb{k} = (k_1,k_2,\ldots,k_n)$ can be minimized 
in $O ({\rm MF}(k n, km))$ time, 
where $k := \max k_i$.
\end{Thm}
Their algorithm is sketched as follows.
Suppose that $S_{k_i} =\{0,1,2,\ldots,k_i\}$ for $1 \leq i \leq n$.
We consider a directed  network ${\cal N}$
with vertex set $U$, edge set $A$, and edge capacity $c: A \to \RR_+$, 
where $U$ consists of source $s$, sink $t$, and 
$v_i^{1}, v_i^{2},\ldots,v_{i}^{k_i}$ $(1 \leq i \leq n)$.
Let $U_i := \{v_i^{1},v_i^{2},\ldots,v_{i}^{k_i}\}$.
%
%
A {\em legal cut} is an $(s,t)$-cut $X$ such that $|X \cap U_i| \leq 1$ for $1 \leq i \leq n$.
There is a natural bijection $\phi$ from $S_{\pmb k}$ to the set of legal cuts, where $\phi$ is given by
\[
\phi(x) := \{s\} \cup \{ v_i^{x_i} \mid 1 \leq i \leq n, x_i \neq 0\} \quad (x = (x_1,x_2,\ldots,x_n) \in S_{\pmb k}).
\]
For an $(s,t)$-cut $X$, let $\check{X}$ denote the legal cut obtained from $X$ 
by deleting $U_i \cap X$ with $|U_i \cap X| \geq 2$ for $i=1,2,\ldots,n$.
We say that network ${\cal N}$ {\em represents} function $f: S_{\pmb k} \to \overline\RR$
if 
\begin{itemize}
\item[(1)] for some constant $K$, it holds
$f(x) = c(\delta \phi(x)) + K$ for every $x \in S_{\pmb k}$, and
\item[(2)] $c(\delta \check{X}) \leq c(\delta X)$ for every $(s,t)$-cut $X$.
\end{itemize}
Suppose that ${\cal N}$ represents $f$.
By (1) and (2), 
the minimum of $f + K$ is equal to the minimum $(s,t)$-cut capacity, and 
some legal cut is a minimum $(s,t)$-cut.
By $\check{X} \subseteq X$ and (2), 
any inclusion-minimal minimum $(s,t)$-cut is necessarily a legal cut.
Recall the fundamental fact in network flow theory
that there is a unique inclusion-minimal mincut $X^*$, 
which is equal to the set of vertices reachable from $s$ in the residual network 
for any maximum $(s,t)$-flow.
Then $x^* := \phi^{-1}(X)$ is a minimizer of $f$.
Namely $f$ is minimized by a single max-flow computation on ${\cal N}$.

As shown in \cite{IWY14},
any nonnegative sum of basic $\pmb k$-submodular functions
is represented by some network ${\cal N}$.
Notice that if $f$ is represented by ${\cal N} = (U,A,c)$, then for constants $\alpha > 0$ and $\beta$,
$\alpha f + \beta$ is represented by ${\cal N}' = (U, A, \alpha c)$, and that
if $f$ and $f'$ are represented 
by networks ${\cal N}$ and ${\cal N'}$ (on the same vertex set $U$), respectively, 
then $f + f'$ is represented by the union of ${\cal  N}$ and ${\cal N'}$.
Also notice that a unary $\pmb k$-submodular function $f$ is decomposed 
to $\theta_a$ and $\epsilon_b$ as (\ref{eqn:unary}).
If $f(0) = \infty$, then there is at most one $u \in S_k$ with $u \in \dom f$, 
and the coordinate $x_i$ of term $f(x_i)$ may be fixed to $u$.
So we can assume that $f(0) \in \dom f$.
Therefore it suffices to construct a network for the following cases of $f$: 
(i$_0$) $f(x) = \infty$ if $x_i = u \neq 0$ and $0$ otherwise, 
(i) 
$f(x) = \epsilon_a(x_i)$ $(a \neq 0)$, (ii) $f(x) = \theta_a (x_i)$, 
(iii) $f(x) = \delta_\sigma (x_i, x_j)$, and (iv) $f(x) = \mu_{a,a'}(x_i, x_j)$.

Case (i$_0$): Consider the network consisting of a single edge $v_i^u t$ 
with infinite capacity.

Case (i): Consider the network consisting of a single edge $v_i^{a}t$ 
with unit capacity.

Case (ii): If $a = 0$, then $f$ is the sum of $\epsilon_{a'}$, and reduces to case (i).
Suppose that $a \neq 0$.
Consider the network 
consisting of edges $s v_i^{a}$ and  $v_i^{j}t$ for $j \in \{1,2,\ldots,k_i\} \setminus \{a\}$.

Case (iii): 
Consider the network consisting of
edges joining $v_i^{u}$ and $v_{j}^{\sigma(u)}$ 
in both directions for $u = 1,2,\ldots,k_i (= k_j)$.

Case (iv): If $a = a' = 0$, then $f$ is 
the sum of two unary $k$-submodular functions, and reduces to case (i). 
So suppose $a \neq 0$.
Consider the network consisting of edges 
$v_j^u v_i^{a}$ for $u \in  \{1,2,\ldots,k_j\} \setminus \{a'\}$ 
and $v_i^u v_j^{a'}$ for $u \in  \{1,2,\ldots,k_i\} \setminus \{a\}$ if $a' \neq 0$.

In all four cases, it is easy to check that the network represents $f$, i.e., 
it satisfies (1) and (2).
In particular, each basic $\pmb k$-submodular function is represented by $O(k)$ edges, where $k = \max_i k_i$.
Thus a nonnegative sum of $m$ basic $\pmb k$-submodular functions is 
represented by a network with $O(kn)$ vertices and $O(km)$ edges.
Hence it can be minimized by $O({\rm MF}(kn,km))$ time.

\subsection{Alternating L-convex functions}\label{subsec:alternating}
Here we introduce the notion of 
an alternating L-convex function on the product of trees.
Alternating L-convex functions 
may be viewed as a natural variant of {\it strongly tree-submodular functions} 
due to Kolmogorov~\cite{Kolmogorov11MFCS}; 
see Remark~\ref{rem:midpoint} for a detailed relation.  
In addition, when $\ZZ^n$ is identified with (the vertex set) 
of the $n$-fold product of a path with infinite length, 
alternating L-convex functions are equal to {\em UJ-convex functions} 
considered by Fujishige~\cite{Fujishige14}.
Also alternating L-convex functions constitutes a useful special class of  
{\em L-convex functions on modular complexes}~\cite{HH13SODA, HH_HJ13}, 
in which a modular complex are taken to be the product of zigzag oriented trees.

Let $T$ be a tree. We will use the following convention:
\begin{quote}
{\em The vertex set of a tree $T$ is also denoted by $T$.}
\end{quote}
The exact meaning will always be clear in this context. 
Regard $T$ as a bipartite graph.
Let $B$ and $W$ be 
the color classes of $T$, 
where a vertex in $B$ is called {\em black} 
and a vertex in $W$ is called {\em white}.
Define a partial order $\preceq$ on $T$ by:
$u \prec v$ if $u$ and $v$ are adjacent with $(u,v) \in W \times B$.
Then the resulting poset has no chain of length $2$;
$B$ and $W$ are the sets of maximal and minimal elements, respectively.
Let $d$ denote the path metric of $T$, i.e., $d(u,v)$ 
is the number of edges in the unique path between $u$ and $v$.
For vertices $u, v \in T$, there uniquely exists 
a pair $(a, b)$ of vertices such that $d(u,v) = d(u,a) + d(a, b) + d(b,v)$, 
$d(u,a) = d(b, v)$, and $d(a,b) \leq 1$.
In particular, $a = b$ or $a$ and $b$ are adjacent.
Define $u \bullet v$ and $u \circ v$ so that
$\{u \bullet v,  u \circ v\} = \{a,b\}$ and $u \circ v \preceq u \bullet v$.
Namely if $a= b$ then $u \bullet v = u \circ v = a = b$. 
Otherwise  $u \bullet v$ and $u \circ v$ are the black and white vertices in $\{a,b\}$, respectively.

Let $n$ be a positive integer. 
We consider the $n$-fold Cartesian product $T^n$ of $T$. 
We will use the $l_{\infty}$-metric on $T^n$, which is also denoted by $d$:
\[
d (x,y) := \max_{1 \leq i \leq n} d(x_i,y_i) \quad (x,y \in T^n).
\]
\begin{Def}
A function $g: T^n \to \barRR$ is 
called {\em alternating L-convex}, or simply, {\em L-convex}
if it satisfies
\begin{equation}\label{eqn:midpoint}
g(x) + g(y) \geq g(x \bullet y) + g( x \circ y) \quad (x,y \in T^n).
\end{equation}
\end{Def}
The defining inequality (\ref{eqn:midpoint}) can be viewed as a variation 
of the discrete midpoint convexity~(\ref{eqn:midpoint_org}); see also Remark~\ref{rem:midpoint}.
As the direct product of posets $T$, 
we regard $T^n$ as a poset, where the partial order is also denoted by $\preceq$.
For $x = (x_1,x_2,\ldots,x_n) \in T^n$, the principal ideal 
${\cal I}(x) := \{ y \in T^n \mid y \preceq x\}$ 
and filter ${\cal F}(x) := \{ y \in T^n \mid y \succeq x\}$ of $x$ are given as follows:
\begin{eqnarray*}
{\cal I}(x) & = & {\cal I}(x_1) \times {\cal I}(x_2) \times \cdots \times {\cal I}(x_n), \\
{\cal F}(x) & = &   {\cal F}(x_1) \times {\cal F}(x_2) \times \cdots \times {\cal F}(x_n), \\
{\cal I}(x_i) & = &
\left\{ 
\begin{array}{ll}
\{x_i\} \cup \{ \mbox{all neighbors of $x_i$} \} & {\rm if}\ x_i \in B, \\
\{x_i\} & {\rm if}\ x_i \in W,
\end{array} \right. \\
{\cal F}(x_i) & = &
\left\{ 
\begin{array}{ll}
\{x_i\} \cup \{ \mbox{all neighbors of $x_i$} \} & {\rm if}\ x_i \in W, \\
\{x_i\} & {\rm if}\ x_i \in B.
\end{array} \right. 
\end{eqnarray*}
Regard  ${\cal I}(x_i)$ as poset $S_{k_i}$ with $k_i := |{\cal I}(x_i)| - 1$ and the minimum element $x_i$, 
and also regard ${\cal F}(x_i)$ as poset $S_{k'_i}$ with $k'_i := |{\cal F}(x_i)|-1$ and the minimum element $x_i$.
Then ${\cal I}(x) \simeq S_{\pmb k}$ for $\pmb{k} = (k_1,k_2,\ldots,k_n)$ 
and ${\cal F}(x) \simeq S_{\pmb k}$ for $\pmb{k} = (k'_1,k'_2,\ldots,k'_n)$.
By this correspondence, 
it is easy to see that $\bullet = \sqcup$ and $\circ = \wedge$ on ${\cal F}(x)$, 
and that $\circ = \sqcup$ and $\bullet = \wedge$ on ${\cal I}(x)$.
Therefore we have:
\begin{Lem}\label{lem:locally_k_submo}
For each $x \in T^n$, 
an L-convex function
$g$ is $\pmb k$-submodular on ${\cal I}(x)$ and on ${\cal F}(x)$.
\end{Lem}
In particular, in the case where $T$ is a star with $k$ leaves and center $x_0 \in W$, 
under the correspondence $T^n = {\cal F}(x_0)^n \simeq {S_k}^n$, 
L-convex functions and $k$-submodular functions are the same.

The following is an analogue of the L-optimality criterion of 
L$^{\natural}$-convex functions~\cite[Theorem 7.5]{MurotaBook} 
(and strongly-tree submodular function~\cite{Kolmogorov11MFCS}):
\begin{Thm}[L-optimality]\label{thm:L-optimality} 
Let $g$ be an L-convex function on $T^n$.
If $x \in \dom g$ is not a minimizer of $g$, then there exists $x' \in {\cal I}(x) \cup {\cal F}(x)$ 
with $g(x') < g(x)$.
\end{Thm}
\begin{proof}
There is $y$ with $g(y) < g(x)$.
Take such $y$ with $d(x, y)$ minimum.
By inequality (\ref{eqn:midpoint}), 
we have $2 g(x) > g(x) + g(y) \geq g(x \bullet y) + g(x \circ y)$.
Thus $g(x \bullet y) < g(x)$ or $g(x \circ y) < g(x)$ holds.
By the minimality, 
it must hold $d(x, y) \leq 1$ (since
$d(y, x \bullet y) \leq \lceil d(x,y)/2 \rceil$ 
and $d(y, x \circ y) \leq \lceil d(x,y)/2 \rceil$). 
Then $x \bullet y \in {\cal F}(x)$, $x \circ y \in {\cal I}(x)$, and the claim holds.
\end{proof}

This directly implies the following descent algorithm, 
which is an analogue of the steepest descent algorithm of 
L$^\natural$-convex functions~\cite[Section 10.3.1]{MurotaBook}.
\begin{description}
\item[Steepest descent algorithm:]
\item[Input:] An L-convex function $g$ and 
a vertex $x \in \dom g$.
\item[Step 1:] Let $y$ be a minimizer of $g$ 
over ${\cal I}(x)\cup {\cal F}(x)$.
\item[Step 2:] If $g(x) = g(y)$, then stop. Otherwise $x := y$, and go to step 1.
\end{description}
By Theorem~\ref{thm:L-optimality}, 
if the algorithm terminates in step 2, then $x$ is a minimizer of $g$.
Also the step 1 is conducted by minimizing $g$ over ${\cal I}(x)$ and $g$ over ${\cal F}(x)$.
%
By Lemma~\ref{lem:locally_k_submo}, $g$ is $\pmb k$-submodular on ${\cal I}(x)$ and on ${\cal F}(x)$, 
and step~1 can be conducted by $\pmb k$-submodular function minimization.
Therefore, if $g$ is represented as 
a sum of $\pmb k$-submodular functions of bounded arity, 
then each iteration is conducted in polynomial time.
To obtain a complexity bound,
we need to estimate the total number of iterations.
In the case of L$^{\natural}$-convex functions,
the number of iterations is bounded by the constant of 
the $l_{\infty}$-diameter of the effective domain~\cite{KS09}.
A recent analysis~\cite{MurotaShioura14} showed that
the number of iterations is {\em exactly} equal to
the minimum of a certain directed analogue of $l_{\infty}$-distance 
from the initial point $x$ to the set of minimizers. 

We will establish an analogous result for our L-convex function.
Let ${\rm opt} (g)$ denote the set of all minimizers of $g$.
Obviously the total number of the iterations  
is at least the minimum $l_{\infty}$-distance $d({\rm opt}(g),x) 
:= \min \{ d(y,x) \mid y \in {\rm opt}(g) \}$
from $x$ to ${\rm opt} (g)$.
This lower bound is almost tight, as follows.
\begin{Thm}\label{thm:bound}
For an L-convex function $g$ on $T^n$ and a vertex $x \in \dom g$,
the total number $m$ of the iterations of the steepest descent algorithm 
applied to $(g, x)$ is at most $d({\rm opt}(g), x) + 2$.
If $g(x) = \min_{y \in {\cal F}(x)} g(y)$ or $g(x) = \min_{y \in {\cal I}(x)} g(y)$, 
then $m = d({\rm opt}(g), x)$.
\end{Thm}
In the case where $x \in B^n$ or $W^n$, 
the condition $g(x) = \min_{y \in {\cal F}(x)} g(y)$ or $g(x) = \min_{y \in {\cal I}(x)} g(y)$
is automatically satisfied.
In fact, \cite{HH_HJ13} announced (a slightly weaker version of) 
this result for general L-convex functions 
on modular complexes. However, the proof needs a deep geometric investigation 
on modular complex, and will be given in a future paper~\cite{HHprepar}.
We give a self-contained proof of 
Theorem~\ref{thm:bound} in Section~\ref{subsec:bound}. 

\subsection{L-extendable functions}\label{subsec:L-extendable} 
Next we introduce the notion of the L-extendability.
This notion was inspired by the idea of {\em ($k$-)submodular relaxation} 
used in~\cite{GK13,  IWY14, Kolmogorov12DAM}.
\begin{Def}
A function $h: B^n \to \barRR$ is called {\em L-extendable} if 
there exists an  L-convex function $g$ on $T^n$ such 
that the restriction of $g$ to $B^n$ is equal to $h$.
\end{Def}
We also define the L-extendability of a function on $T^n$ 
via the edge-subdivision.
The {\em edge-subdivision} $T^*$ of $T$ 
is obtained by adding a new vertex $w_{uv}$ 
for each edge $e = uv$, and 
replacing $e$ in $T$ by two edges $uw_{uv}, w_{uv} v$.
The new vertex $w_{uv}$ is called the {\em midpoint} of $uv$.
The original $T$ is a subset of $T^*$ and is 
one of the color classes of $T^*$.
In $T^*$, vertices in $T$ are supposed to be black and 
midpoints are supposed to be white.
\begin{Def}
A function $h: T^n \to \barRR$ is called {\em midpoint L-extendable}, 
or simply, {\em L-extendable} 
if there exists an  L-convex function $g$ on $(T^*)^n$ such 
that the restriction of $g$ to $T^n$ is equal to $h$.
\end{Def}
We call $g$ an {\em L-convex relaxation} of $h$.
If $\min_{x \in T^n} g(x) = \min_{x \in (T^*)^n} h(x)$, 
then $g$ is called an {\em exact} L-convex relaxation.
In fact, any L-convex function admits an exact L-convex relaxation, 
and is (midpoint) L-extendable~\cite{HHprepar}; 
see Remark~\ref{rem:midpoint} for related arguments.
We will see that vertex-cover problem and multiway cut problem 
admits $k$-submodular relaxation (L-convex relaxation in our sense).
This means that it is NP-hard to minimize L-extendable functions in general.
However L-extendable functions have several useful properties.
The main results in this section are following three properties of L-extendable functions.
These properties will play crucial roles in the proximity scaling algorithm 
for minimum cost multiflow problem in Section~\ref{sec:multiflow}. 
Proofs of the three theorems are given in Section~\ref{subsec:proof}.

The first property is an optimality criterion analogous to Theorem~\ref{thm:L-optimality}:
\begin{Thm}[Optimality]\label{thm:optimality_h}
Let $h: B^n \to \barRR$ be an L-extendable function.
For $x \in \dom h$, 
if $x$ is not a minimizer of $h$, 
then there exists $y \in B^n$ such that $d(x,y) \leq 2$ 
and $h(y) < h(x)$.
\end{Thm}
The second property is so-called the {\em persistency}.
This notion was introduced by Kolmogorov~\cite{Kolmogorov12DAM} 
for bisubmodular relaxation,  
and was extended to $k$-submodular relaxation~\cite{GK13, IWY14}.
The persistency property says that  from a minimizer $x$ of a relaxation $g$, 
we obtain a minimizer $y$ of $h$ by {\em rounding} each white component of $x$ 
to an adjacent (black) vertex.
\begin{Thm}[{Persistency}]\label{thm:persistency}
Let $h:B^n \to \barRR$ be an L-extendable function  
and $g: T^n \to \barRR$ its L-convex relaxation.
For a minimizer $x$ of $g$,
then there is a minimizer $y$ of $h$ with $y \in {\cal F}(x) \cap B^n$.
\end{Thm}
The third one is a proximity theorem. 
The proximity theorem of 
L$^{\natural}$-convex function $g:\ZZ^n \to \barRR$ says 
that for any minimizer $x$ of $g$ over $(2\ZZ)^n$, 
there is a minimizer $y$ of $g$ with $\|x - y\|_{\infty} \leq n$~
(Iwata-Shigeno~\cite{IS02}; 
see \cite[Theorem 7.6]{MurotaBook} and \cite[Theorem 20.10]{FujiBook}).
We establish an analogous result for L-extendable functions.
\begin{Thm}[{Proximity}]\label{thm:proximity}
Let $h: T^n \to \barRR$ be a midpoint L-extendable function, 
and let $x$ be a minimizer of $h$ over $B^n$.
Then there exists a minimizer $y$ of $h$ with $d(x, y) \leq 2 n$.
In addition, if $h$ admits an exact L-convex relaxation, 
then there exists a minimizer $y$ of $h$ with $d(x, y) \leq n$
\end{Thm}

\begin{Ex}[Vertex cover]
As noted in \cite{Kolmogorov12DAM}, vertex cover problem is 
a representative example
admitting a bisubmodular relaxation (L-convex relaxation in our sense).
Let $G = (V,E)$ be an undirected graph with (nonnegative) cost $a$ on $V = \{1,2,\ldots,n\}$.
A {\em vertex cover} is a set $X$ of vertices meeting every edge.
The vertex cover problem asks to find a vertex cover $X$ of minimum cost $a(X)$.
The well-known IP formulation of this problem is:
Minimize $\sum_{i \in V} a(i) x_i$ over $x \in \{0,1\}^{n}$ satisfying 
$x_i + x_j \geq 1$ for $ij \in E$.
Define $I_{\geq 1}: \RR \to \overline\RR$ by $I_{\geq 1}(z) = 0$ if $z \geq 1$ 
and $\infty$ otherwise, and define $\omega: \{0,1\}^n \to \overline{\RR}$ by
$\omega(x) := \sum_{i \in V} a(i) x_i + \sum_{ij \in E} I_{\geq 1}(x_i + x_j)$.
Then the vertex cover problem is the minimization of $\omega$.
This function $\omega$ is midpoint L-extendable 
(if $\{0,1\}$ is identified with the vertex set of the graph of single edge).
Indeed, the natural extension $\bar \omega: \{0,1/2,1\}^n \to \overline{\RR}$
is a bisubmodular relaxation
(if $\{0,1/2,1\}$ is identified with $S_2$ with $0 \succ 1/2 \prec 1$).

The {\em submodular vertex cover problem}~\cite{IN09} 
is to minimize submodular function 
$f: \{0,1\}^n \to \overline{\RR}$ over $x \in \{0,1\}^n$ 
satisfying $x_i + x_j \geq 1$ for $ij \in E$. 
Namely this is the minimization of $\omega$ defined by 
$\omega(x) = f(x) + \sum_{ij \in E} I_{\geq 1} (x_i + x_j)$.
Again $\omega$ is midpoint L-extendable.
Indeed 
a function $x \mapsto (f(\lceil x \rceil) + f(\lfloor x \rfloor))/2 + \sum_{ij \in E} I_{\geq 1} (x_i + x_j)$
is a bisubmodular relaxation of $\omega$.
\end{Ex}

\begin{Rem}\label{rem:midpoint}
We can consider several variants of 
the discrete midpoint convexity inequality and associated discrete convex functions.
Suppose that each edge of $T$ has an orientation.
Let ${\rm mid}: T \times T \to T^*$ be defined by: 
${\rm mid}(p,q)$ is the unique vertex $u \in T^*$
with $d(p,u) = d(u,q)$ and $d(p,u) + d(u,q) = d(p,q)$.
For $u \in T^* \setminus T$, 
let $\overline u$ and $\underline u$ denote
the vertices of $T$ such that 
$\overline u \underline u$ is an edge with midpoint $u$, 
and is oriented from $\overline u$ to $\underline u$.
For $u \in T$, let $\overline u = \underline u := u$.
Extend these operations to operations on $T^n$ in componentwise, 
as above.
Consider function $g$ satisfying
\begin{equation}\label{eqn:midpoint'}
g(x) + g(y) \geq g( \overline{{\rm mid} (x, y)}) + g( \underline{{\rm mid} (x,y)}) \quad (x,y \in T^n).
\end{equation}
In the case where $T$ is a path on $\ZZ$ obtained 
by joining $i$ and $i+1$ and by orienting $i \to i+1$ $(i \in \ZZ)$,  
the above inequality (\ref{eqn:midpoint'}) 
is equal to (\ref{eqn:midpoint_org}), and $g$ is L$^\natural$-convex.
In the case where there is a unique sink in $T$, i.e., $T$ is a rooted tree, 
the operations $\overline{\rm mid}$ and $\underline{\rm mid}$ are equal, respectively, 
to $\sqcup$ and $\sqcap$ in the sense of~\cite{Kolmogorov11MFCS}, 
and $g$ is strongly tree submodular.
Also notice that alternating L-convex functions correspond to the zigzag orientation.

So different orientations of $T$ define different classes of 
discrete convex functions.
Theory of L-extendable functions 
captures all these discrete convex functions by the following fact:
\begin{myitem}
If $g: T^n \to \overline\RR$ 
satisfies (\ref{eqn:midpoint'}), then $g$ is midpoint L-extendable,  
where its exact L-convex relaxation $\bar g: (T^*)^n \to \overline\RR$ is given by
\[
\bar g(u) :=  ( g(\overline u ) + g(\underline u ))/2 \quad (u \in (T^*)^n).
\] \label{eqn:fact} \vspace{-0.5cm}
\end{myitem}
\noindent 
We will give the proof of this fact in \cite{HHprepar} (since it is bit tedious).
In particular the minimization of $g$ over $T^n$ can be solved by 
the minimization of $\bar g$ over $(T^*)^n$.
Instead of $g$, we can apply our results to $\bar g$
(and obtain results for original $g$). 
This is another reason why 
we consider alternating L-convex functions and L-extendable functions.
\end{Rem}

\subsection{$2$-separable convex functions}\label{subsec:2-separable}
In this section, we introduce a special class of L-extendable functions, 
called {\em 2-separable convex functions}.
This class is an analogue of
a class of functions $f$ on $\ZZ^n$ represented as the following form:
 \begin{equation}\label{eqn:2separable_ZZ}
 \sum_{i} f_i (x_i) + \sum_{i,j} g_{ij} ( x_i -  x_j ) + \sum_{i,j} h_{ij} ( x_i + x_j ) \quad (x \in \ZZ^n),
 \end{equation}
where $f_i, g_{ij}$, and $h_{ij}$ are $1$-dimensional convex functions on $\ZZ$.
Hochbaum~\cite{Hochbaum02} 
considers minimization of functions with this form,  
and provides a unified framework to 
NP-hard optimization problems 
admitting half-integral relaxation and 2-approximation algorithm.
Recall (\ref{eqn:typical}) that
a function without terms $h_{ij}(x_i+ x_j)$ 
is a representative example of L$^\natural$-convex functions.
It is known that the half-integral relaxation 
of (\ref{eqn:2separable_ZZ})
can be efficiently minimized by a maximum flow algorithm~\cite{Hochbaum02}; 
also see~\cite{AHO04, KS09}. 

In this section, we show analogous results:
a 2-separable convex function admits an L-convex relaxation each of 
whose local $\pmb k$-submodular function is 
a sum of basic $\pmb k$-submodular functions.
Hence this L-convex relaxation can be efficiently minimized 
by successive applications of max-flow min-cut computations. 
Moreover, for some special cases,
a solution of the L-convex relaxation can be rounded 
to a 2-approximation solution of the original 2-separable convex function.
  
We start with the (one-dimensional) convexity on a tree.
A function $h$ on $\ZZ$ is said to be {\em nondecreasing} if
$
\Delta h(t) := h(t) - h(t-1) \geq 0$ for $t \in \ZZ$, an is said to be {\em convex} if
$
\Delta^2h(t) := h(t+1) - 2h(t) + h(t-1) \geq 0$ for $t \in \ZZ$, and 
is said to be {\em even} if 
$(h(t - 1) + h(t + 1))/2 = h(t)$ for every odd integer $t$.
We can naturally define convex functions on a tree $T$.
It should be noted that this notion of convexity  was considered in 
the classical literature of facility location analysis~\cite{DFL76, Kolen, TFL83}.
For $u,v \in T$ and an integer $t$ with $0 \leq t \leq d(u,v)$, 
let $[u,v]_t$ denote the unique vertex $s$ satisfying $d(u,v) = d(u,s) + d(s,v)$ and 
$t = d(u,s)$.
A function $h$ on $T$ is said to be {\em convex}
if for any vertices $u,v$ in $T$, 
a function on $\ZZ$, defined by $t \mapsto h([u,v]_t)$ for $0 \leq t \leq d(u,v)$ 
(and $t \mapsto +\infty$ otherwise),  is convex. 
\begin{Lem}\label{lem:tree_convexity}
For a function on $T$, the convexity, L-convexity, and L-extendability are equivalent.
For convex functions $f,g$ on $T$ and $\alpha, \beta \in \RR_+$,   
$\alpha f+ \beta g$ is convex,
and $\max (f, g)$, defined by $u \mapsto \max \{f(u), g(u)\}$, is convex.
\end{Lem}
Let $h$ be a function on $\ZZ$. 
For vertices $z,w \in T$, 
we consider three functions $h_T, h_{T;z}, h_{T;z,w}$ defined by
\begin{eqnarray*}
h_T (u,v) & := & h(d(u,v))   \quad (u,v \in T), \\
h_{T;z} (u) & := & h(d(u,z))  \quad (u \in T), \\
h_{T;z,w} (u,v) & := & h(d(u,z) + d(v,w))  \quad (u,v \in T).
\end{eqnarray*}
We will see below that $h_{T;z}$ is L-convex, 
and the other two functions are (midpoint) L-extendable; 
notice that they are not L-convex in general.
\begin{Thm}\label{thm:2separable} 
Let $h$ be a non-decreasing convex function on $\ZZ$, 
and let $z,w$ be vertices of $T$.
\begin{itemize}
\item[{\rm (1)}] 
$h_{T;z}$ is convex on $T$.
\item[{\rm (2)}] Suppose that $h$ is even.
Then $h_{T}$ is L-convex. Moreover, for $(u,v) \in T \times T$, 
function $h_T$ on ${\cal F}(u) \times {\cal F}(v)$ is 
a sum of basic $\pmb k$-submodular functions. 
Namely,  
for $(s,t) \in {\cal F}(u) \times {\cal F}(v)$, 
it holds
\begin{equation*}
h_T (s,t) - h (D) =
\left\{ \begin{array}{ll} 
\Delta h(1) \delta (s,t) & {\rm if}\ u = v \in W, \\
\Delta h(D) \theta_{a}(s) & {\rm if}\ u \in W, v \in B, \\
\Delta  h(D) \theta_{b}(t) & {\rm if}\ u \in B, v \in W, \\
 \Delta  h(
 D) (\theta_{a}(s) + \theta_{b}(t))  + 
\Delta^2  h(D) \mu_{a,b}(s,t)  & {\rm if}\ u,v \in W: u \neq v,\\
 0 & {\rm if}\ u,v \in B.
\end{array}\right. 
\end{equation*}
where $D := d(u,v)$, and
$a$ and $b$ are vertices in ${\cal F}(u)$ and in ${\cal F}(v)$ nearest to $v$ and $u$, respectively. 

\item[{\rm (3)}] Suppose that $h$ is even, and
$z,w$ belong to the same color class.
Then $h_{T;z,w}$ is L-convex. %
Moreover, for $(u,v) \in T \times T$, function $h_{T;z,w}$ 
on ${\cal F}(u) \times {\cal F}(v)$
is a sum of basic $\pmb k$-submodular functions.  Namely, for $(s,t) \in {\cal F}(u) \times {\cal F}(v)$, it holds
\begin{equation*}
h_{T;z,w} (s,t) - h(D) 
 = 
\left\{ \begin{array}{ll}
\Delta  h (D) \theta_a(s)  & {\rm if}\  u \in W, v \in B, \\
\Delta  h (D) \theta_b(t)  & {\rm if}\  u \in B, v \in W, \\
\Delta  h(D) (\theta_a(s) + \theta_b(t))  + 
\Delta^2  h(D) \mu_{a,b}(s,t)
& {\rm if}\ u,v \in W, \\
0  & {\rm if}\ u,v \in B,
\end{array}\right. 
\end{equation*}
where $D:= d(u,z) + d(v,w)$, 
and $a$ and $b$ are vertices in ${\cal F}(u)$ and  ${\cal F}(v)$ nearest to $z$ and $w$, respectively.
\end{itemize}
\end{Thm}
The local expressions of $h_{T}$ and $h_{T;z,w}$ 
on ${\cal I}(u) \times {\cal I}(v)$
are obtained by replacing roles of $B$ and $W$.

A function $\omega$ on $T^n$ is said to be {\em $2$-separable L-convex} 
if $\omega$ is a sum of functions given in Theorem~\ref{thm:2separable}:
\begin{equation}\label{eqn:omega}
\omega (x) := \sum_{i} f_i (x_i) + \sum_{i,j} g_{ij} (d(x_i, x_{j})) 
+ \sum_{i,j} h_{ij} (d(x_i, z_i) + d(x_j, w_j)) \quad (x \in T^n),
\end{equation}
where each $f_i$ is a convex function on $T$, each $g_{ij}$ and $h_{ij}$ are  nondecreasing even convex function on $\ZZ$, 
and $z_i$ and $w_j$ are vertices in the same color class.
A function $\omega$ on $B^n$ is said to be {\em $2$-separable convex} 
if $\omega$ is the form of (\ref{eqn:omega}) where 
each $g_{ij}$ and $h_{ij}$ are (not necessarily even) 
nondecreasing convex functions on $\ZZ$.
A $2$-separable convex function on $B^n$ is L-extendable, and
its L-convex relaxation $\bar \omega$ on $T^n$ is explicitly given by
\begin{equation}
\bar \omega (x) := \sum_{i} f_i (x_i) + \sum_{i,j} \bar g_{ij} (d(x_i, x_{j})) 
+ \sum_{i,j} \bar h_{ij} (d(x_i, z_i) + d(x_j, w_j)) \quad (x \in T^n),
\end{equation}
where $\bar g_{ij}$ and $\bar h_{ij}$ are even functions obtained from $g_{ij}$ and $h_{ij}$ by replacing $g_{ij}(z)$ and $h_{ij}(z)$ 
by $(g_{ij}(z - 1) + g_{ij}(z+1))/2$ and $(h_{ij}(z - 1) + h_{ij}(z+1))/2$, respectively, 
for each odd integer $z$. 
A function $\omega$ on $T^n$ is also said be $2$-separable convex 
if $\omega$ is the form of (\ref{eqn:omega}) where 
each $g_{ij}$ and $h_{ij}$ are (not necessarily even) 
nondecreasing convex functions on $\ZZ$.
Then $\omega$ is midpoint L-extendable.

By Theorem~\ref{thm:2separable}, 
the L-convex relaxation $\bar \omega$ is locally a sum of basic $k$-submodular functions. 
Hence $\bar \omega$ is efficiently minimized 
by successive applications of maximum flow (minimum cut) computations.
We will see in Section~\ref{sec:multiflow} that
2-separable convex functions arise 
from minimum cost multiflow problems.

A {\em convex multifacility location function} is a special 
2-separable convex function represented as
\begin{equation}
\omega (x) = \sum_{i,j} f_{ij} (d(x_i, z_j)) + \sum_{i,j} g_{ij} (d(x_i, x_{j})) \quad (x \in B^n),
\end{equation}
where $f_{ij}$ and $g_{ij}$ are nonnegative-valued 
nondecreasing convex functions, and $z_j$ are black vertices.
In this case, we take an L-convex relaxation
\begin{equation}\label{eqn:bar_omega}
\bar \omega (x) = \sum_{i,j} \bar f_{ij} (d(x_i, z_j)) + \sum_{i,j} \bar g_{ij} (d(x_i, x_{j})) \quad (x \in T^n).
\end{equation}
For the case where all terms are 
linear functions $b_{ij} d(x_i, z_j)$, $c_{ij} d(x_i, x_j)$ 
with nonnegative coefficients $b_{ij}$, $c_{ij}$, 
the problem of minimizing (\ref{eqn:bar_omega}) is known as 
a {\em multifacility location problem} on a tree; see \cite{Kolen, TFL83} 
and also its recent application to computer vision~\cite{FPTZ10, GK13}, where
a faster algorithm in \cite{GK13} is 
applicable to the case where only $g_{ij}$ are linear 
(since our notion of convexity
is the same as {\em $T$-convexity} in \cite{GK13} 
for the case of uniform edge-length).

There is a natural rounding scheme from $T^n$ to $B^n$. 
In some cases, we can construct a good approximate solution for $\omega$ 
from a minimizer of $\bar \omega$.
For $y \in B$ and $x \in T^n$, let $x_{\rightarrow y}$ denote the vertex $z \in B^n$
such that for each $i=1,2,\ldots,n$,
$z_i = x_i$ if $x_i \in B$, and
$z_i$ is the unique neighbor of $x_i$ close to $y$ (i.e., $d(x_i,y) = 1 + d(z_i,y)$) if $x_{i} \in W$.
\begin{Thm}\label{thm:approx}
Let $\omega: B^n \to  \RR$ be a $2$-separable convex function consisting of $m$ terms, 
and $\bar \omega: T^n \to \RR$ be its L-convex relaxation.
\begin{itemize}
\item[{\rm (1)}] For a given $x \in B^n$, there is 
an $O(d({\rm opt}(\bar \omega), x) {\rm MF}(k n, k m))$ time algorithm to find 
a global minimizer $x^*$ of $\bar \omega$ over $T^n$, 
where $k$ is the maximum degree of $T$.
\item[{\rm (2)}] Suppose that $\omega$ is a convex multifacility location function.
For any $y \in B$, the rounded solution $(x^*)_{\rightarrow y}$
is a $2$-approximate solution of $\omega$.
\end{itemize}
\end{Thm}
\begin{Ex}[Multiway cut]\label{rem:multiway}
As discussed in \cite{IWY14}, multiway cut problem is 
a representative example of NP-hard problems admitting a $k$-submodular relaxation 
(an L-convex relaxation in our sense). 
Let $G = (V,E)$ be an undirected graph with a set $S  \subseteq V$ 
of terminals and an edge-capacity $c:E \to \RR_+$.
Let $V \setminus S = \{1,2,\ldots,n\}$.
A {\em multiway cut} ${\cal X}$ is a partition of $V$ such that 
each part contains exactly one terminal.
The capacity $c({\cal X})$ of ${\cal X}$ is the sum of $c(ij)$ over all edges $ij$ 
whose ends $i$ and $j$ belong to distinct parts in ${\cal X}$. 
The multiway cut problem in $G$ is the problem of finding 
a multiway cut with minimum capacity.

The multiway cut problem is formulated as minimization of 
a multifacility location function on a star.
Let $T$ be a star with leaf set $S$ and center vertex $0$. 
Suppose that $B = S$ and $W = \{0\}$.
Consider the following 2-separable convex function minimization: 
\begin{eqnarray}\label{eqn:multiway}
\mbox{Min.}  && \frac{1}{2} \sum_{1 \leq i \leq n} \sum_{s \in S: si \in E} 
c(si) d(s, x_i)   +  \frac{1}{2} \sum_{ij \in E: 1 \leq i,j \leq n} c(ij) d (x_i, x_j), \\
\mbox{s.t.} && (x_1,x_2,\ldots,x_n) \in B^n. \nonumber  
\end{eqnarray}
This is equivalent to the multiway cut problem on $G$.
To see this,  for a multiway cut ${\cal X}$,
define $x = (x_1,x_2,\ldots,x_n)$ by: $x_i := s$ if $i$ and $s$ 
belong to the same part of ${\cal X}$.
Then the objective of the resulting $x$ is equal to the cut capacity of ${\cal X}$.
Conversely, for a solution $x$ of (\ref{eqn:multiway}), 
define $X_s$ by the set of all vertices $i$ with $x_i = s$ for $s \in S$.
Then ${\cal X} := \{X_s\}_{s \in S}$ is a multiway cut 
whose capacity is equal to the objective value of $x$. 

An L-convex relaxation problem relaxes the constraint $x \in B^n$ into 
$x \in T^n = (B \cup \{ 0\})^n$.  
If $T^n$ is identified with ${S_k}^n$,  
then this is a $k$-submodular function minimization, where 
$d$ is equal to $\delta$.
An optimum $x^*$ of the relaxed problem can be efficiently obtained 
by $(s, S \setminus \{s\})$-mincut computations for $s \in S$; 
see Section~\ref{subsub:LC}.
Take some $s \in B$. Consider $y := (x^*)_{\rightarrow s}$.
This 2-approximation is 
essentially equal to the classical 2-approximation algorithm 
of multiway cut problem~\cite{DJPSY94}; see \cite[Algorithm 4.3]{Vazirani}.
\end{Ex}

\subsection{Proofs}\label{subsec:proof}
In this section, we give proofs of results.
We will often use the following variation of $\pmb k$-submodularity inequality.
\begin{Lem}\label{lem:k-submo'}
If $f$ is a $\pmb k$-submodular function on $S_{\pmb k}$, then
\begin{equation}\label{eqn:k-submo'}
f(x) + f(y) \geq f(x \wedge y) + \frac{1}{2} f(x \sqcup (x \sqcup y)) +  \frac{1}{2} f((x \sqcup y) \sqcup y)
\quad (x,y \in S_{\pmb k}).
\end{equation}
\end{Lem}
\begin{proof}
Observe that $u \wedge (u \sqcup v)$ is $0$ if $0 \neq u \neq v \neq 0$
 and is $u$ otherwise $(u,v \in S_k)$.
From this we have
\begin{eqnarray*}
((x \sqcup y) \wedge x) \sqcup (( x \sqcup y) \wedge y) & = & x \sqcup y,  \\
((x \sqcup y) \wedge x) \wedge (( x \sqcup y) \wedge y) & = & x \wedge y.
\end{eqnarray*}
Similarly $u \sqcup (u \sqcup v)$ is $v$ if $u = 0$ and is $u$ otherwise. From this we have
\[
(x \sqcup (x \sqcup y)) \sqcup  ((x \sqcup y) \sqcup y) = 
(x \sqcup (x \sqcup y)) \wedge ((x \sqcup y) \sqcup y) = x \sqcup y.
\]
Thus we have
\begin{eqnarray*}
&& f((x \sqcup y) \wedge x) + f(( x \sqcup y) \wedge y) \geq f(x \wedge y) + f(x \sqcup y), \\
&& f(x) + f(x \sqcup y)  \geq f(x \wedge (x \sqcup y)) + f(x \sqcup (x \sqcup y) ), \\
&& f(x \sqcup y) + f(y) \geq f((x \sqcup y) \wedge y ) + f((x \sqcup y) \sqcup y), \\
&& f(x \sqcup (x \sqcup y)) + 
f((x \sqcup y) \sqcup y) \geq 2 f(x \sqcup y).
\end{eqnarray*}
Adding the first three inequalities and one half of the forth inequality, 
we obtain (\ref{eqn:k-submo'}).
\end{proof}
The binary operation $(x,y) \mapsto x \sqcup (x \sqcup y)$ 
plays important roles in the subsequent arguments.
We note the following properties:
\begin{myitem1}
\item[(1)] $x \sqcup (x \sqcup y) \succeq x$. 
\item[(2)] If $y$ has no zero component, then so does $x \sqcup (x \sqcup y)$.
\item[(3)] $(x \sqcup (x \sqcup y)) \sqcup y = x \sqcup y$ and 
$(x \sqcup (x \sqcup y)) \sqcup (x \sqcup y) = x \sqcup (x \sqcup y)$. \label{eqn:sqcup}
\end{myitem1}\noindent
These properties immediately follow from the behavior of each component:
\[
u \sqcup (u \sqcup v) = \left\{ 
\begin{array}{ll}
v & {\rm if}\ u = 0, \\
u & {\rm otherwise}, 
\end{array}\right. \quad (u,v \in S_{k}).
\]

\subsubsection{Proof of Theorem~\ref{thm:bound}}\label{subsec:bound}
Let $g$ be an L-convex function on $T^n$, 
and let $x$ be a vertex in $\dom g$. We first show the latter part. 
Suppose (w.l.o.g.) that
\begin{equation}\label{eqn:g(x)=min_g(y)}
 g(x) = \min_{y \in {\cal F}(x)} g(y).
\end{equation}
Let $x = x^0,x^1,x^2,\ldots,x^m$ be a sequence of vertices 
in generated by the steepest descent algorithm applied to $(g,x)$.
Then it holds
\begin{equation}\label{eqn:zigzag}
x = x^0 \succ x^1 \prec x^2 \succ x^3 \prec x^4 \succ \cdots.
\end{equation}
Indeed, $x^0 \succ x^1$ follows from (\ref{eqn:g(x)=min_g(y)}). 
Also $x^{i} \prec x^{i+1} \prec x^{i+2}$ (or $x^{i} \succ x^{i+1} \succ x^{i+2}$) never occurs. 
Otherwise $x^{i+2} \in {\cal F}(x^i)$ and $g(x^{i+2}) < g(x^{i+1}) < g(x^i)$; 
this is impossible from the definition of the steepest descent algorithm. 
\begin{Lem}\label{lem:geodesic}
For $z \in {\cal I}(x^k) \cup {\cal F}(x^k)$ with $g(z) < g(x^k)$, we have
\[
d(x,z) = k+1 \quad (k=1,2,\ldots).
\]
\end{Lem}

\begin{proof}
By (\ref{eqn:zigzag}), it holds that $x^{k-1}, x^{k+1} \in {\cal F}(x^{k})$ if $k$ is odd 
and $x^{k-1}, x^{k+1} \in {\cal I}(x^{k})$ if $k$ is even. 
We use the induction on $k$.
Then we can assume that $d(x^{k-1}, x) = k-1$ and 
$d(x^{k}, x) = k$.
Suppose for the moment that $k$ is odd.
We are going to show that $d(z, x) = k+1$.
By Lemma~\ref{lem:locally_k_submo}, $g$ is $\pmb k$-submodular on ${\cal F}(x^k)$. 
Notice that both $x^{k-1}$ and $z$ belong to ${\cal F}(x^k)$.
By Lemma~\ref{lem:k-submo'}, we have
\begin{equation}
g(z) + g(x^{k-1}) \geq g(z \wedge x^{k-1}) + \frac{1}{2} g(z \sqcup (z \sqcup x^{k-1})) 
+  \frac{1}{2} g((z \sqcup x^{k-1}) \sqcup x^{k-1}). 
\end{equation}

Since $x^k$ is a minimizer of $g$ over ${\cal I}(x^{k-1})$, 
we have
\begin{equation*}
g(z \wedge x^{k-1}) \geq g(x^{k}) > g(z).
\end{equation*} 
Hence we necessarily have
\begin{equation}
2 g(x^{k-1}) > g(z \sqcup (z \sqcup x^{k-1})) +  g((z \sqcup x^{k-1}) \sqcup x^{k-1}).
\end{equation}
This implies
$g(x^{k-1}) > g(z \sqcup (z \sqcup x^{k-1}))$ 
or $g(x^{k-1}) > g((z \sqcup x^{k-1}) \sqcup x^{k-1})$.
The second case is impossible.
This follows from:  $x \preceq  (z \sqcup x) \sqcup x$ (see (\ref{eqn:sqcup})~(1)) and (\ref{eqn:g(x)=min_g(y)}) for $k=1$, and
$x^{k-2} \preceq x^{k-1} \preceq  (z \sqcup x^{k-1}) \sqcup x^{k-1}$ and
$g(x^{k-1}) = \min_{y \in {\cal F}(x^{k-2})} g(y)$ for $k > 1$.
Thus we have
\begin{equation}\label{eqn:=>>}
g((z \sqcup x^{k-1}) \sqcup x^{k-1}) \geq g(x^{k-1}) > g(z \sqcup (z \sqcup x^{k-1})).
\end{equation}
Let $z' := x^{k-1} \wedge (z \sqcup (z \sqcup x^{k-1}))$. 
Then $x^{k} \preceq z' \preceq x^{k-1}$.
By Lemma~\ref{lem:k-submo'}, we have
\begin{equation*}
g(z \sqcup (z \sqcup x^{k-1})) + g(x^{k-1}) \geq g(z') + 
\frac{1}{2} g(z \sqcup (z \sqcup x^{k-1})) + \frac{1}{2} g( (z \sqcup x^{k-1}) \sqcup x^{k-1}),
\end{equation*}
where we use (\ref{eqn:sqcup})~(3) to obtain the second and third terms in the right hand side.
By (\ref{eqn:=>>})
we have
$
g(x^{k-1}) > g(z').
$
Notice $z' \in {\cal I}(x^{k-1})$.
By induction, we have
\begin{equation}\label{eqn:|p-q'|=i}
d(x,z') = k.
\end{equation}
Hence we can take an index $j$ 
with $d(x_j,z'_j) = k$. 
By $x^k_j \preceq z'_j \preceq x^{k-1}_j$ 
(in ${\cal F}(x^{k}_j)$), 
we have $x^k_j = z'_j \prec x^{k-1}_j$ and $d(x_j,x^{k-1}_j) = k-1$. 
Since $x^{k}_j = ((x^{k-1}_j \sqcup z_j) \sqcup z_j) \wedge x^{k-1}_j$, 
we have $x^{k-1}_j \neq z_j \neq x^{k}_j$; otherwise
$x^{k-1}_j = z_j$ or $x^{k}_j = z_j$ implies a contradiction $x^{k-1}_j = x^{k}_j$.
Thus $x^{k-1}_j$ and $z_j$ are distinct neighbors of $x^k_j$ in $T$
with $d(x_j, x^{k-1}_j) = k-1$ and  $d(x_j, x^{k}_j) = k$.
Since $T$ is a tree, we have $d(x_j, z_j) = k + 1$, and $d(x, z) = k + 1$.
The argument for even $k (\geq 2)$ is same; 
reverse the partial order $\preceq$.
\end{proof}

Let ${\rm opt}(g)$ be the set of minimizers of $g$, 
and let $m^* := \min_{z \in {\rm opt}(g)} d(x,z)$.
Since $x^m$ is optimal, we have $m^* \leq m$.
Our goal is to show $m = m^*$.
Let $\tilde g$ be a function defined 
by $\tilde g(y) := g(y)$ if $d(x,y) \leq m^*$ 
and $\tilde g(y) := \infty$ otherwise.
Then $\tilde g$ is also L-convex 
since $d(x, y) \leq m^*$ and $d(x, y') \leq m^*$ imply
$d(x, y \bullet y') \leq m^*$ and $d(x, y \circ y') \leq m^*$. 
The subsequence $x = x^1,x^2,\ldots,x^{m^*}$ is also 
obtained by applying the steepest descent algorithm to 
$\tilde g$ from $x$.
By Lemma~\ref{lem:geodesic}, no vertex $z$ with 
$d(x,z) > m^*$ is produced.
Hence $x^{m^*}$ is 
necessarily a minimizer of $\tilde g$, 
and is also a minimizer of $g$. Thus $m = m^*$. 
This completes the latter part of the proof of Theorem~\ref{thm:bound}.
The former part is now immediately obtained.
Suppose that $x$ does not satisfy (\ref{eqn:g(x)=min_g(y)}).
But $x^1$ always satisfies $g(x^1) = \min_{y \in {\cal F}(x^1)} g(y)$ or $g(x^1) = \min_{y \in {\cal I}(x^1)} g(y)$.
The sequence $(x^1,x^2,\ldots,x^m)$ is also obtained by the steepest descent algorithm.
By the latter claim and the triangle inequality, 
we have $m - 1 = d(x^1, \opt(g)) \leq d(x, \opt(g)) + 1$.

\begin{Prop}\label{prop:sequence}
Let $g$ be an L-convex function on $T^n$.
For $y \in {\rm opt}(g)$ and $x \in \dom g$ with $g(x) = \min_{y \in {\cal F}(x)} g(y)$, 
there is a sequence $x = x^0,x^1,x^2,\ldots, x^m = y$ such that
\begin{itemize}
\item[{\rm (1)}] $m = d(x,y)$,
\item[{\rm (2)}] $g(x^{i}) > g(x^{i+1})$ for 
$i < d({\rm opt}(g),x)$ and $g(x^{i}) = g(x^{i+1})$ for $i \geq d({\rm opt}(g),x)$, and
\item[{\rm (3)}] $g(x^{i+1}) = \min \{ g(z) \mid z \in {\cal I}(x^{i})\}$ for even $i$ and 
$g(x^{i+1}) = \min \{ g(z) \mid z \in {\cal F}(x^{i})\}$ for odd~$i$.
\end{itemize}
\end{Prop}
\begin{proof}
Take a sufficiently small $\epsilon > 0$. Consider the function $g'$ defined by
\[
g'(x) = g(x) + \epsilon \sum_{i=1}^n d(x_i, y_i) \quad (x \in T^n).
\]
By Theorem~\ref{thm:2separable} (that will be proved independently), the second term is L-convex, 
and hence $g'$ is also L-convex, and has the unique minimizer $y$.
Apply the steepest descent algorithm to $(g',x)$.
By Theorem~\ref{thm:bound}, we obtain a sequence $x = x^0,x^1,\ldots,x^m = y$ with (1), (2), and (3) for $g'$.
Since $\epsilon$ is sufficiently small, 
any steepest direction $x^{i+1}$ for $g'$ at $x^i$ is also a steepest direction for $g$.
Thus the sequence also satisfies (1),(2), and (3) for $g$.
\end{proof}

\subsubsection{Proof of Theorem~\ref{thm:persistency}}

Let $x$ be a minimizer of an L-convex relaxation $g$ of an L-extendable function $h$.
Take a minimizer $y \in B^n$ of $h$ such that $d(x,y)$ is minimum.
We can take a sequence $y = y^0,y^1,y^2,\ldots,y^m = x$ 
satisfying the conditions in Proposition~\ref{prop:sequence}, where $m = d(x,y)$.
If $m = 1$, then $x \in {\cal I}(y)$ ($y \in {\cal F}(x)$) and we are done.
Suppose (indirectly) $m \geq 2$.
Let $z := y^1$ and $w := y^2$.
Notice that $y,w \in {\cal F}(z)$.
Applying Lemma~\ref{lem:k-submo'} to $\pmb k$-submodular function $g$ on ${\cal F}(z)$, we have
\[
g(y) + g(w) \geq g(y \wedge w) + \frac{1}{2} g(y \sqcup (y \sqcup w)) + \frac{1}{2} g(w \sqcup (w \sqcup y)).
\]
Since $y$ is a maximal element in ${\cal F}(z)$ and $y \preceq y \sqcup (y \sqcup w)$, 
we have $y = y \sqcup (y \sqcup w)$, and $g(y \sqcup (y \sqcup w)) = g(y)$.
Also  it holds $g(y \wedge w) \geq g(z)  \geq g(w)$ (by Proposition~\ref{prop:sequence}~(3)).
This implies 
\[
g(y) \geq g(w \sqcup (w \sqcup y)).
\]
Here $w':= w \sqcup (w \sqcup y)$ is also maximal (in ${\cal F}(z)$) and has no zero components (see (\ref{eqn:sqcup}) (2)).
Thus $w'$ belongs to $B^n$, and is also a minimizer of $h$.
Since $d(w',w) =1$ (by $w \preceq w'$) 
and $d(w,x) = m - 2$, we have $d(x,y) > d(x,w')$. 
A contradiction to the minimality.

\subsubsection{Proof of Theorem~\ref{thm:optimality_h}}
Let $g$ be an L-convex relaxation of $h$.
Suppose that $x$ is not a minimizer of $h$.
Then $x$ is not a minimizer of $g$.
By the L-optimality criterion (Theorem~\ref{thm:L-optimality}), 
there is $z \in {\cal I}(x)$ such that $g(z) < g(x)$.
If $z$ is a minimizer of $g$, 
then by Theorem~\ref{thm:persistency} 
there is a minimizer $y \in B^n \cap {\cal F}(z)$ of $h$, as required.
Suppose that $z$ is not a minimizer of $g$.
There is $w \in {\cal F}(z)$ such that $g(w) < g(z)$.
By Lemma~\ref{lem:k-submo'} with $x \sqcup (x \sqcup w) = x$, we have
\[
g(x) + g(w) \geq g(x \wedge w) + \frac{1}{2} g(x) + \frac{1}{2} g(w \sqcup (w \sqcup x)).
\]
Notice $g(w) < g(x \wedge w)$.
Hence $g(x) > g(w \sqcup (w \sqcup x))$, and
$w \sqcup (w \sqcup x)$ is a required vertex in $B^n$ (by (\ref{eqn:sqcup})~(2)).

\subsubsection{Proof of Theorem~\ref{thm:proximity}}

We start with preliminary arguments.
For a quarter integer $u \in \ZZ/4$,  
define half-integers $[u]_1, [u]_{1/2} \in \ZZ/2$ by
\begin{equation*}
[u]_1  :=   \left\{ 
\begin{array}{ll}
  u      & {\rm if} \ u \in \ZZ/2, \\
\mbox{the integer nearest to $u$} & {\rm otherwise},  
\end{array}\right.
\end{equation*}
\begin{equation*}
[u]_{1/2} := \left\{ 
\begin{array}{ll}
u & {\rm if}\  u \in \ZZ/2, \\
\mbox{the non-integral half-integer nearest to $u$} & {\rm otherwise}.
\end{array}\right.
\end{equation*}
\begin{Lem}\label{lem:1/4}
For $u,v \in \ZZ/4$, we have
\begin{eqnarray*}
&& \lfloor u - v \rfloor \leq [u]_1 - [v]_1 \leq \lceil u - v \rceil, \\
&& \lfloor u - v \rfloor \leq [u]_{1/2} - [v]_{1/2} \leq \lceil u - v \rceil.
\end{eqnarray*}
\end{Lem}
\begin{proof}
It suffices to consider 
the cases  $([u]_1,[v]_1) =  (u \pm 1/4, v)$, $(u, v \pm 1/4)$,  
$(u + 1/4, v - 1/4)$, or $(u - 1/4, v + 1/4)$.
Suppose that the first two cases occur.
Then $u - v$ is not a half-integer, and hence 
$\lfloor u - v \rfloor \leq u - v \pm 1/4 \leq \lceil u - v \rceil$.
Consider the last two cases.
Then $u \in \ZZ - 1/4$ and $v \in \ZZ + 1/4$ or $u \in \ZZ + 1/4$ and $v \in \ZZ - 1/4$.
Hence $u - v \in \ZZ + 1/2$. 
This means that $u - v$ is not an integer but a half-integer. 
Thus $\lfloor u - v \rfloor \leq u - v \pm 1/2 \leq \lceil u - v \rceil$.
The second inequality follows from the same argument.
\end{proof}

Let $x,y \in T^n$.
Let $P_i$ denote the unique path connecting $x_i$ and $y_i$ in $T$.
We regard vertices of $P_i$ as integers $0,1,\ldots, d_i := d(x_i,y_i)$ 
by the following way.
Associate vertex $u$ in $P_i$ with integer $d(x_i,u) \in \{0,1,\ldots, d_{i}\}$.
Then $x_i = 0$ and $y_i = d_i$.
Similarly,  let $P_i^*$ denote the unique path connecting $x_i$ and $y_i$ in $T^{*}$.
Associate the midpoint of each edge $uv$ in $P_i$ with half-integer 
$(u + v)/2 (= (d(x_i,u) + d(x_i,v))/2)$.
Then the vertices of the product $P := P_1 \times P_2 \times \cdots \times P_n$
are integer vectors $z$ with $0 \leq  z_i \leq d_i$, and the vertices of 
$P^* := P^*_1 \times P^*_2 \times \cdots \times P^*_n$ are half-integer vectors 
$z$ with $0 \leq z_i \leq d_i$.
Under this correspondence, it holds
\begin{equation*}
z \bullet z' = [ (z+ z')/2 ]_1, \quad z \circ z' = [  (z+ z')/2]_{1/2} \quad (z,z' \in P^* \subseteq (\ZZ/2)^n ),
\end{equation*}
where $[\cdot]_{1}$ and $[\cdot]_{1/2}$ are extended on $(\ZZ/4)^n$ 
in componentwise.

For $i \in \{1,2,\ldots,n\}$, let $e_i$ denote the $i$-th unit vector, 
and let $\pi_{i} := e_1 + e_2 + \cdots + e_i$.
We can assume that $d_1 \geq d_2  \geq \cdots \geq d_n$.
Then $y$ is represented as
\begin{equation}\label{eqn:expression}
y = x + \sum_{i=1}^{n} (d_{i}- d_{i+1}) \pi_{i},
\end{equation}
where we let $d_{n+1} = 0$.
Let $((x,y))$ and $((x,y))^*$ be the sets of integral points and half-integral points, respectively, 
in the polytope
\begin{equation}
Q(x,y) := \{  z \in \RR^n \mid 0 \leq z_{i} - z_{i+1} \leq d_{i} - d_{i+1} \ (1 \leq i \leq n) \},
\end{equation}
where we let $z_{n+1} := 0$.
Observe that the polytope $Q(x,y)$ is the set of points $z$ represented as
\begin{equation}\label{eqn:represented}
z = x +  \sum_{i=1}^{n} \alpha_i \pi_{i}
\end{equation}
for $\alpha_i \in [0, d_{i} - d_{i+1}]$ $(i=1,2,\ldots,n)$.
Note that this representation is unique.
\begin{Lem}\label{lem:z+z'}
For $z,z' \in ((x,y))^*$, 
both $[(z + z')/2]_1$ and $[(z + z')/2]_{1/2}$ belong to $((x,y))^*$.
\end{Lem}
\begin{proof}
We show that 
half-integer vectors $[(z +z')/2]_1$ and $[(z + z')/2]_{1/2}$
belong to $Q(x,y)$. 
By convexity, $w := (z+ z')/2 \in (\ZZ/4)^n$ belong to $Q(x,y)$. Hence
\[
0 \leq w_i - w_{i+1} \leq d_{i} - d_{i+1}.
\]
Notice that $d_{i} - d_{i+1}$ is integral. By Lemma~\ref{lem:1/4}, we have
\[
0 \leq [w_i]_1 - [w_{i+1}]_1 \leq d_{i} - d_{i+1}, \quad 
0 \leq [w_i]_{1/2} - [w_{i+1}]_{1/2} \leq d_{i} - d_{i+1}.
\]
This means that both $[w]_1$ and $[w]_{1/2}$ belong to $Q(x,y)$.
\end{proof}
Let $h: T^n \to \overline{\RR}$ be an L-extendable function, 
and let $g: (T^*)^n \to \overline{\RR}$ be its L-convex relaxation.
The discrete midpoint convexity inequality on $P^*$ is given by
\begin{equation}\label{eqn:midpoint_P*}
g(z) + g(z') \geq g( [(z+z')/2]_{1}) + g([(z+z')/2 ]_{1/2}) \quad (z,z' \in P^*). 
\end{equation}
In particular, for $i,j \in \{1,2,\ldots,n\}$, we have
\begin{eqnarray}
g(\pi_{j}/2) + g(\pi_{i}/2 + \pi_{j}) & \geq & g(\pi_{j}) + g( \pi_{i}/2 + \pi_{j}/2), \label{eqn:pi1} \\
g(0) + g(\pi_{i}/2 + \pi_{j}/2) & \geq & g(\pi_{i}/2) + g(\pi_{j}/2), \label{eqn:pi2} \\
g(\pi_i/2 + \pi_j/2) + g(\pi_i + \pi_j) & \geq &  g(\pi_i + \pi_j/2) + g(\pi_i/2 + \pi_j), \label{eqn:pi3}\\
g(\pi_i/2) + g(\pi_i + \pi_j/2)  & \geq &  g(\pi_i/2 + \pi_j/2) + g(\pi_i). \label{eqn:pi4}
\end{eqnarray}
For example,  
$(\pi_i+3\pi_j)/4 = \pi_{i} + 3(e_{i+1} + e_{i+2} + \cdots + e_{j})/4$ for $i < j$
and $(\pi_i+3\pi_j)/4 = \pi_{j} + (e_{j+1} + e_{j+2} + \cdots + e_{i})/4$ for $i > j$.
Thus $[(\pi_i + 3 \pi_j)/4]_1 = \pi_j$ and $[(\pi_i + 3 \pi_j)/4]_{1/2} = \pi_i/2 + \pi_j/2$.
From (\ref{eqn:midpoint_P*}),
we see the first equality. 
The remaining are obtained in the same way.
\begin{Lem}\label{lem:dom}
For $x,y \in \dom h$, it holds $((x,y)) \subseteq \dom h$.
\end{Lem}
\begin{proof}
We use the induction on $d(x,y) = k$.
We can assume that $\dom h$ belongs to $((x,y))$.
Indeed, modify $g$ and $h$ so that they take $\infty$ 
on points not belonging to $Q(x,y)$.
By Lemma~\ref{lem:z+z'}, 
$g$ is still L-convex, and hence is an L-convex relaxation of $h$.
In particular $h$ is L-extendable.
We may assume that one of $x,y$, say $x$, is not a minimizer of $h$ 
(by adding a 2-separable L-convex function 
$z  \mapsto \sum_{i=1}^n d(z_i, y_i)$ to $h, g$ if necessarily).

Use expression (\ref{eqn:expression}) to represent $x, y$. 
By induction it suffices to show that 
each $\pi_i$ with  $d_i > d_{i+1}$ belongs to $\dom h$, 
since $d(\pi_i, y) = k-1$ and 
$y = \pi_i +  (d_{i} - d_{i+1} - 1) \pi_i + \sum_{j \neq i} (d_j - d_{j+1}) \pi_j$.
By Theorem~\ref{thm:optimality_h} with expression (\ref{eqn:represented}), 
there is $j$ such that $d_{j} > d_{j+1}$ and $h(x) = h(0) > h(\pi_{j}) < \infty$.
Consider an index $i \neq j$ with $d_i > d_{i+1}$. 
We show that $\pi_i \in \dom h$.
By induction for $(\pi_j, y)$, we have $\pi_i + \pi_j \in \dom h \subseteq \dom g$.
By applying the midpoint convexity (\ref{eqn:midpoint_P*})
for $g$ at $(0, \pi_j)$ and at $(\pi_j, \pi_i + \pi_j)$, 
we have $\pi_j/2, \pi_j + \pi_i/2 \in \dom g$.
By (\ref{eqn:pi1}),
we have $\pi_i/2 + \pi_j /2  \in \dom g$.
Similarly, by (\ref{eqn:pi2}),
we have $\pi_i/2 \in \dom g$.
By (\ref{eqn:pi3})
we have $\pi_i + \pi_j /2 \in \dom g$.
Finally, by (\ref{eqn:pi4}),
we have $\pi_i \in \dom h$, as required.
\end{proof}
The essence of the proximity theorem is the following, 
where this lemma may be viewed as an analogue of \cite[(20.37)]{FujiBook}.
\begin{Lem}\label{lem:pi_i}
For $x \in \dom h$ and a minimizer $y$ of $h$ 
with $d(x, \opt(h)) = d(x,y)$, we have 
\[
h(x) > g(x + \pi_i /2) > h(x + \pi_{i}) \quad (i: d_{i} - d_{i+1} \geq 2).
\]
\end{Lem}
\begin{proof}
As above, we can assume that $\dom h$ belongs to $((x,y))$, 
and  $\dom h = ((x,y))$ (by Lemma~\ref{lem:dom}).
Then $y$ is a unique minimizer of $h$ over $((x,y))$.
We use the induction on $d(x,y)$.
Consider first the case where $y = x + (d_i - d_{i+1}) \pi_i$ for some $i$ 
with $d_i - d_{i+1} \geq 2$.
By Theorem~\ref{thm:optimality_h} with induction, 
we have $h(0) > h(\pi_i) > h(2\pi_i)$.
By $h(0) + h(\pi_i) = g(0) + g(\pi_i) \geq 2 g(\pi_i/2)$ 
we have $h(0) > g(\pi_i/2).$
By $h(\pi_i) + h(2 \pi_i) = g(\pi_i) + g(2\pi_i) \geq 2 g(3\pi_i/2)$, 
we have $g(\pi_i) > g(3\pi_i/2)$.
By $g(\pi_i/2) + g(3\pi_i/2) \geq 2 g(\pi_i)$, 
we have $g(\pi_i/2) > g(\pi_i) = h(\pi_i)$. 
Thus $h(0) > g(\pi_i/2) > h(\pi_i)$ holds.

Consider the general case.
Take $i$ with $d_i - d_{i+1} \geq 2$.
We may assume that there is $j \neq i$ with $d_j - d_{j+1} \geq 1$.
By induction for $(\pi_j, y)$, we have
\begin{equation}\label{eqn:j>i+j}
g(\pi_j) = h(\pi_j) > g(\pi_j + \pi_i/2) > g(\pi_i + \pi_j) = h(\pi_i + \pi_j). 
\end{equation}
By (\ref{eqn:j>i+j}) and (\ref{eqn:pi1}), 
we have $g(\pi_j/2) > g(\pi_i/2 + \pi_j/2)$, 
and, by (\ref{eqn:pi2}), $g(0) > g(\pi_{i}/2)$.
Similarly, by  (\ref{eqn:j>i+j}) and (\ref{eqn:pi3}), 
we have $g(\pi_i/2 + \pi_j/2) > g(\pi_i + \pi_j/2)$, 
and, by (\ref{eqn:pi4}), $g(\pi_i/2) > g(\pi_i)$.
Thus we have $h(0) = g(0) > g(\pi_i/2) > g(\pi_i) = h(\pi_i)$, as required.
\end{proof}
We are ready to prove Theorem~\ref{thm:proximity}.
Let $h$ be a midpoint L-extendable function on $T^n$, and 
let $x$ be a minimizer of $h$ over $B^n$.
Let $y$ be a minimizer of $h$ over $T^n$ with $d(x, {\rm opt}(h)) = d(x,y)$.
We can assume that $h(y) < h(x)$.
Consider $((x,y))$ as above.
Since $x$ is a minimizer of $h$ over $B^n$,
we have
$h(x) \leq h(x + \alpha \pi_i)$ for an even integer $\alpha$.
On the other hand, by Lemma~\ref{lem:pi_i}, 
we have $h(x) > h(x + \alpha \pi_{i})$ for 
$\alpha \in 1,2,\ldots, d_{i} - d_{i+1} - 1$ if $d_i - d_{i+1} \geq 2$.
This means that $d_{i} - d_{i+1} \geq 3$ is impossible. 
Thus $d_i - d_{i+1} \leq 2$ for each $i$.
Hence $d(x,y) \leq 2 n$.

Consider the case where $h$ admits as an exact L-convex relaxation.
In the proof of Lemma~\ref{lem:pi_i}, 
we can assume that $y$ is also a unique minimizer of $g$ (by perturbing $h,g$ if necessarily).
Consequently the statement of Lemma~\ref{lem:pi_i} 
holds for index $i$ with $d_i - d_{i+1} \geq 1$.
Therefore, in the above argument,  $d_{i} - d_{i+1} \geq 2$ is impossible. 
Thus we obtain the latter statement of Theorem~\ref{thm:proximity}.

\subsubsection{Proof of Lemma~\ref{lem:tree_convexity}}
It suffices to consider the case where $T$ is a path.
Hence $T$ is naturally identified with $\ZZ$, and $B = 2\ZZ$.
Suppose that $f$ is convex on $\ZZ$.
Then it is easy to see that $f(u)  + f(v) \geq f( \lfloor (u + v)/2 \rfloor) 
+ f( \lceil (u + v)/2 \rceil) = f( u \circ v) + f(u \bullet v)$.
Thus $f$ is (alternating) L-convex. The converse is also easy.
If $f:\ZZ \to \RR$ is convex, then $\bar f: \ZZ/2 \to \RR$ 
defined by $u \mapsto (f(\lfloor u \rfloor) + f( \lceil u\rceil))/2$ 
is also convex (and L-convex) on $T^*$, and $f$ is L-extendable.
The converse is also easy: the restriction of convex function on $\ZZ/2$ to $\ZZ$
is also convex on $\ZZ$. 
The latter part is straightforward to be verified.

\subsubsection{Proof of Theorem~\ref{thm:2separable}}

We begin with preliminary arguments on the convexity of a tree.
Let $\bar T$ be the set of all formal combinations of vertices  of
form $\lambda u + \mu v$, where $uv$ is an edge, and 
$\lambda$ and $\mu$ are nonnegative reals with $\lambda + \mu = 1$.
Informally speaking, 
$\bar T$ is a ``tree" obtained by filling the ``unit segment" to each edge.
We can naturally regard $T$ and 
$T^*$ as subsets of $\bar T$ (by $T^{*} \ni w_{uv} \mapsto (1/2) u + (1/2) v$).
Also the metric $d$ on $T$ is extended to $\bar T$ as follows.
For two points $p= \lambda u + \mu v$, $p' = \lambda' u' + \mu' v' \in \bar T$,
if $(u,v) = (u',v')$, then $d(p,p') := |\lambda - \lambda'| = |\mu - \mu'|$.
Otherwise we can assume that 
$d(v,v') = d(v,u) + d(u,u') + d(u',v')$.
Define $d(p,p') := \mu + d(u,u') + \mu'$.

For points $p,q \in \bar T$ and $t \in [0,1]$, 
there is a unique point $r$, denoted by $p \circ_{t} q$, 
such that $d(p,q) = d(p, r) + d(r,q)$, $d(p,r) := t d(p,q)$, 
and $d(r,q) := (1- t) d(p,q)$. 
Consider the Cartesian product $\bar T^n$.
For $x,y \in \bar T^n$, define 
$x \circ_t y := (x_1 \circ_t y_1, x_2 \circ_t y_2,\ldots, x_n \circ_t y_n)$.
A function $f$ on $\bar T^n$ is said to be {\em convex} 
if it satisfies
\[
(1 - t) f(x) +  t f(y) \geq f(x \circ_t y) \quad (t \in [0,1], x,y \in \bar T^n).
\] 
In the case where $T$ is a path of infinite length,
$\bar T$ is isometric to $\RR$, and 
this convexity coincides with the ordinary Euclidean convexity.
An old theorem in location theory, due to Dearing, Francis, and Lowe~\cite{DFL76}, says that
the distance function $d$ is convex on $\bar T^2$.
\begin{Lem}[\cite{DFL76}]\label{lem:DFL}
$d$ is convex on $\bar T^2$.
\end{Lem}
We note local expressions of functions 
$(s,t) \mapsto h(d(s,t))$ and $(s,t) \mapsto h(d(s,a) + d(t,b))$.
\begin{Lem}\label{lem:123}
Let $h$ be an even function on $\ZZ$ and let $u, v \in W$.
\begin{itemize}
\item[{\rm (1)}] For $s,t  \in {\cal F}(u)$, we have
\begin{equation*}
 h(0) + \Delta h(1) \delta (s,t) =  h(d(s,t)).
\end{equation*} 
\item[{\rm (2)}] For $s,a \in {\cal F}(u)$, we have
\[
h(1) + \Delta h(1) \theta_{a}(s) 
=
\left\{ \begin{array}{ll}
h(d(s,a)) & {\rm if}\ a \neq u,\\
h(d(s,a) + 1) & {\rm if}\ a = u.
\end{array}\right.
\]
\item[{\rm (3)}] For $(s,t),(a,b) \in {\cal F}(u) \times {\cal F}(v)$, we have
\begin{eqnarray*}
&& h(2) + \Delta h(2) (\theta_a(s) + \theta_b(t)) + \Delta^2  h(2) \mu_{a,b}(s,t)  \\[0.5em]
&& \ = \left\{ \begin{array}{ll}
h(d(s,a) + d(t,b) + 2) & {\rm if }\ (a,b) = (u,v), \\
  h(d(s,a) + d(t,b) +1) & {\rm if }\ a = u, b \neq v\ {\rm or}\ a \neq u, b = v,\\
  h(d(s,a)+ d(t,b)) & {\rm if}\ a \neq u, b \neq v.
 \end{array}\right.  \\
\end{eqnarray*}
\end{itemize}
\end{Lem}
\begin{proof}
(1). Observe $d(s,t) = \delta (s,t)$.
Thus $d(s,t) = 0$ implies $h(0) + \Delta h(1) \cdot 0 = h(0)$.
Also $d(s,t) = 1$ implies $h(0) + \Delta h(1) \cdot 1 = h(1)$, and 
$d(s,t) = 2$ implies that  $h(0) + \Delta h(1) \cdot 2 
= 2 h(1) - h(0) = h(2)$ (by the evenness of $h$).

(2). Consider the case where the right hand side is equal to $h(2)$.
Then $a \neq s \neq u$ must hold.
Therefore the left hand side is $h(1) + \Delta h(1) \cdot  1 = h(2)$.
Suppose that the right hand side is equal to $h(1)$.
Then $s = u$ holds.
Therefore the left hand side is $ h(1) + \Delta h(1) \cdot  0 = h(1)$.
Suppose that the right hand side is equal to $h(0)$.
Then $s = a \neq u$ holds.
Therefore the left hand side is $h(1) + \Delta h(1) \cdot  (-1) = h(0)$. 

(3). Consider the case where the right hand side is equal to $h(4)$.
Then $u \neq s \neq a$ and  $v \neq t \neq b$ must hold.
The left hand side is equal to
$h(2) + \Delta h(2) \cdot 2 +  \Delta^2 h(2) \cdot 2= - h(2) + 2h(3) = h(4)$.
Consider the case where the right hand side is equal to $h(3)$.
Then $u \neq s \neq a$ and $v = t$ or $u = s$ and $v \neq t \neq b$ must hold.
The left hand side is equal to
$h(2) + \Delta h(2) \cdot 1 +  \Delta^2 h(2) \cdot 1= h(3)$.

Consider the case where the right hand side is equal to $h(2)$.
Then $s = u$ and $t = v$ must hold.
The left hand side is equal to
to $h(2) + \Delta h(2) \cdot 0 +  \Delta^2  h(2) \cdot 0=  h(2)$.

Consider the case where the right hand side is equal to $h(1)$.
Then $s = a = u$ and $t = b \neq v$,   $s = a \neq u$ and $t = b  = v$, 
$s = a \neq u$  and $b \neq t = v$, or $a \neq  s = u$ and $t = b \neq v$.
The left hand side is equal to
$h(2) + \Delta h(2) \cdot (-1) +  \Delta^2 h(2) \cdot 0= h(1)$.

Consider the case where the right hand side is equal to $h(0)$.
Then $s= a \neq u$ and $t = b \neq v$.
The left hand side is equal to
$h(2) + \Delta h(2) \cdot (-2) +  \Delta^2 h(2) \cdot 0= - h(2) + 2 h(1) = h(0)$.
\end{proof}

\paragraph{Proof of (1).}
Take any vertex $x$ of $T$ and its two distinct neighbors $y,y'$.
It suffices to show $h(d(y,z)) + h(d(y',z)) \geq 2 h(d(x,z))$.
Observe that $\{ d(y,z), d(y',z)\} = \{ d(x,z) + 1, d(x,z) -1\}$ or 
$d(y,z) = d(y',z) = d(x,z) +1$ holds.
For the first case, $h(d(y,z)) + h(d(y',z)) = h(d(x,z) +1) + h(d(x,z)-1) 
\geq 2h(d(x,z))$ by the convexity of $h$.
For the second case, 
$h(d(y,z)) + h(d(y',z)) \geq h(d(x,z)) + h(d(x,z)) = 2h(d(x,z))$ by the monotonicity of $h$.

\paragraph{Proof of (2).}
Extend $h: \ZZ \to \RR$ to $\ZZ/2 \to \RR$ by
$h(z) := h(z)$ if $z \in \ZZ$ and 
$h(z) := (h(z - 1/2) + h(z+ 1/2))/2$ otherwise.
%
%
Take $(u,v), (u',v') \in T^2$. We are going to show the discrete midpoint convexity 
for $h_{T}$:
\[
h(d(u,v))  +  h (d(u',v')) \geq h(d(u \circ u', v \circ v'))  
+ h (d(u \bullet u', v \bullet v')).
\]

Since $h$ is convex, we have
\[
h(d(u,v)) + h (d(u',v')) \geq 2 h \left(\frac{d(u,v) + d(u',v')}{2} \right).
\]
%
By Lemma~\ref{lem:DFL}, we have 
$d(u,v) + d(u',v') \geq 2 d(u \circ_{1/2} u', v \circ_{1/2} v')$.
Since $h$ is nondecreasing, we have
\[
2 h \left(\frac{d(u,v) + d(u',v')}{2} \right) \geq 2 h (d(u \circ_{1/2}  u', v \circ_{1/2} v')).
\]

\begin{Clm}
$
2  h (d(u \circ_{1/2} u', v \circ_{1/2} v')) = h(d(u \circ u', v \circ v')) 
+ h(d(u \bullet u', v \bullet v')). 
$
\end{Clm}
\begin{proof}
Let $\bar u := u \circ_{1/2} u'$ and $\bar v := v \circ_{1/2} v'$.
The claim is obvious when $\bar u = \bar v$ or
both $\bar u$ and $\bar v$ belong to $T$. 
So we consider the other cases.

Case 1: $\bar u\in T$ and $\bar v \not \in T$.
Then $d(\bar u, \bar v)$ is a half-integer, and hence we have
\[
2  h(d(\bar u, \bar v)) =  h(d(\bar u, \bar v) - 1/2) 
+  h(d(\bar u, \bar v) + 1/2). 
\]
Since $\bar u \in T$, we have $\bar u = u \circ u' = u \bullet u'$, and
\[
\{ d(\bar u, v \circ v'), d(\bar u, v \bullet v')\} = \{ d(\bar u, \bar v) - 1/2, d(\bar u, \bar v) +1/2\}. 
\]
Hence we have the claim.

Case 2: $\bar u \not \in T$ and $\bar v \not \in T$.
We can take edges $ss'$ and $tt'$ of $T$ such that $\bar u$ and $\bar v$ 
are the midpoints of $ss'$ and $tt'$, respectively.
In particular, $d(\bar u,s) = d(t, \bar v) =d(\bar u, s') = d(t', \bar v) = 1/2$.
We can assume that 
\[
 d(\bar u, \bar v) = d(\bar u, s)  + d(s,t) + d(t, \bar v).
\]
Then $d(\bar u, \bar v) = d(s',t) = d(s,t')$, and $d(s',t') = d(s,t) + 2$.
Suppose that $d(s,t)$ is odd.
Then $s$ and $t$ belong to different color classes; 
so $(s', t) = (u \circ u', v \circ v')$ and $(s,t') = (u \bullet u', v \bullet v')$ 
or $(s, t') = (u \circ u', v \circ v')$ and $(s',t) = (u \bullet u', v \bullet v')$. 
Thus $d(\bar u, \bar v) = d(u \circ u', v \circ v') = d(u \bullet u', v \bullet v')$, 
and the claim is true.
Suppose that $d(s,t)$ is even. 
Then $(s, t) = (u \circ u', v \circ v')$ and $(s', t') = (u \bullet u', v \bullet v')$
or  $(s', t') = (u \circ u', v \circ v')$ and $(s, t) = (u \bullet u', v \bullet v')$.
Thus we have
\[
\{  d(u \circ u', v \circ v'), d(u \bullet u', v \bullet v')  \} = \{d(s,t), d(s,t) + 2\}.
\]
By $d(\bar u, \bar v) = d(s,t) + 1$ that is odd, 
we have
\[
2  h (d(\bar u, \bar v)) = h(d(s,t)) + h(d(s,t) + 2).
\]
Thus we have the claim.
\end{proof}
Therefore $h_T$ is L-convex.
We verify the latter part of (2). 
For the case for $u = v \in W$, we obtain the formula from Lemma~\ref{lem:123}~(1).
For other cases, we have
\begin{equation}
h(d(s,t)) = h(d(s,a) + d(a,b) + d(b, t)).
\end{equation}
If $u \in W,v \in B$, then $t = b =v$, and $d(a,b) = D - 1$ is even.
By Lemma~\ref{lem:123}~(2), 
we obtain $h_T(s,t) = h(d(s,a) + D -1) = h(D) + \Delta h(D) \theta_a(s)$.
The argument for the case $u \in B,v \in W$ is the same.
Suppose that $u,v \in W$ with $u \neq v$.
Then $d(s,a) + d(t,b) + D - 2$, and $d(a,b) = D - 2$ is even.
By Lemma~\ref{lem:123}~(3) applied to  $h_T(s,t) = h(d(s,a) + d(t,b) + D -2)$, 
we obtain the required formula. 

\paragraph{Proof of (3).}  
We show the discrete midpoint convexity for $h_{T;z,w}$.
Take vertices $u,v,u',v' \in T$.
Then we have
\begin{eqnarray*}
&& h(d(u,z) + d(v,w)) + h(d(u',z) + d(v',w))  \\
&& \geq 2 h \left( \frac{d(u,z) + d(v,w) + d(u',z) + d(v',w)}{2} \right) \\
&& \geq 2 h ( d(u \circ_{1/2} u',z) + d(v \circ_{1/2} v',w)),
\end{eqnarray*}
where we use convexity of $h$ in the first inequality, 
and use the monotonicity and Lemma~\ref{lem:DFL} in the second.
\begin{Clm}
$2 h ( d(u \circ_{1/2} u',z) + d(v \circ_{1/2} v',w)) 
=  h_{T;z,w} (u \bullet u',  v \bullet v') 
+  h_{T;z,w} (u \circ u', v \circ v')$.
\end{Clm}
\begin{proof}
Let $\bar u := u \circ_{\mbox{\tiny$1/2$}} u'$ and $\bar v := v \circ_{\mbox{\tiny$1/2$}} v'$.

Case 1: $\bar u \in T$ and  $\bar v \in T$.
In this case, $\bar u = u \bullet u' = u \circ u'$ 
and $\bar v = v \bullet v' = v \circ v'$ hold, 
and hence the claim holds.

Case 2: $\bar u  \not \in T$ and 
$\bar v  \in T$.
In this case, there is an edge $st$ of $T$ 
such that $\bar u$ is the midpoint of $st$.
We can assume that $d(z,s) = d(z,t) + 1$ and $d(z, \bar u) = d(z,t) + 1/2$.
Since $d(\bar u,z) + d(\bar v, w)$ is not an integer, we have
\begin{eqnarray*}
 2 h( d(\bar u,z) + d(\bar v, w)) & = & h(d( \bar u ,z) + d(\bar v, w) +1/2) + 
 h(d(\bar u,z) + d(\bar v,w) -1/2) \\
&= & h(d(s,z) + d(\bar v,w)) + h(d(t, z) + d(\bar v,w)). 
\end{eqnarray*}
Then the claim follows from 
$\bar v = v \bullet v' = v \circ v'$, and
$\{ u \bullet u', u \circ u'\} = \{s, t\}$.

Case 3: $\bar u \not \in T$ and $\bar v \not \in T$.
Let $ab$ denote the edge such that $\bar u$ is the midpoint of $ab$, 
and let $st$ denote the edge such that $\bar v$ is the midpoint of $st$.
We can assume that 
$(a,b) = (u \bullet u', u \circ u')$ and $(s,t) = (v \bullet v', v \circ v')$.

Case 3.1: $d(a,z) - d(b,z) = d(s,w) - d(t, w) \in \{-1,1\}$. 
Since $z,w$ have the same color and $a,s$ have the same color, 
$d(\bar u,z) + d(\bar v, w) = d(a,z) + d(s,w) \pm 1$ must be odd.
Hence we have
\begin{eqnarray*}
2 h( d(\bar u,z) + d(\bar v,w)) 
& = &  h(d(\bar u,z)+ d(\bar v, w) - 1) +  h(d(\bar u,z)+ d(\bar v,w) + 1) \\
& = &  h(d(a,z)+ d(s, w)) + h(d(b,z)+ d(t,w))
\end{eqnarray*}
as required.

Case  3.2: $d(a,z) - d(b,z) \neq d(s,w) - d(t,w)$.
In this case,  
$d(\bar u,z) + d(\bar v,w) = d(a,z) + d(s,w) = d(b, z)+ d(t, w)$, 
and hence we have the claim.
\end{proof}
We show the latter part of (3).
It holds that $d(s,z) + d(t,w) = d(s,a) + d(t,b) + d(a,z) + d(b,w)$.
Suppose that $u \in W, v \in B$. 
Then ${\cal F}(v) = \{v\}$, $b = v = t$, and $D$ is odd.
Suppose that $z \neq u$. Then $a \neq u$, and 
$d(s,z) + d(t,w) = d(s,a) + D - 1$. 
Thus by Lemma~\ref{lem:123}~(2), 
we have $h_{T;z,w} (s,t) = h(d(s,a) + D - 1) = 
h(D) + \Delta  h(D) \theta_a(s)$.
Suppose that $z = u$. Then $a = u$, and 
$d(s,z) + d(t,w) = d(s,a) + 1 + (D - 1)$.
By Lemma~\ref{lem:123}~(2), we have $h_{T;z,w} (s,t) = h(d(s,a) + 1+ D - 1) = 
 h(D) + \Delta h(D) \theta_a(s)$.
The argument for  $u \in B, v \in W$ is similar.

Suppose that $u,v \in W$.
If $u \neq z$ and $v \neq w$, then $a \neq u$ and $b \neq v$.
Also $d(s,z) + d(t,w) = d(s,a) + d(t,b) + D - 2$, and $D$ is even. 
Apply Lemma~\ref{lem:123}~(3) to $h_{T;z,w} (s,t) = h(d(s,a) + d(t,b) + D - 2)$, 
we obtain the formula.

If $u = z$ and $v \neq w$ or $u \neq z$ and $v = w$, then 
$u = a$ and $v \neq b$ or $u \neq a$ and $v = b$, 
and we have $d(s,z) + d(t,w) = d(s,a) + d(t,b) + D - 1$ with $D$ even. 
By Lemma~\ref{lem:123}~(2) to $h_{T;z,w} (s,t) = h(d(s,a) + d(t,b) + 1 + (D - 2))$,
we obtain the formula.
If $u = z$ and $v = w$, then $u = a$, $v = b$, and
$d(s,z) + d(t,w) = d(s,a) + d(t,b)$ with $D = 0$.
By Lemma~\ref{lem:123}~(2) to $h_{T;z,w} (s,t) = h(d(s,a) + d(t,b))$, we obtain the formula.

\subsubsection{Proof of Theorem~\ref{thm:approx}}

(1). By Theorem~\ref{thm:2separable}, 
all functions $h_T, h_{T,z}, h_{T,z,w}$ are locally basic $k$-submodular.
Hence, by Theorem~\ref{thm:IWY14}, 
the minimization of $\bar \omega$ over ${\cal I}(x)$ and ${\cal F}(x)$ 
can be done in $O({\rm MF}(kn, km))$ time.
By Theorem~\ref{thm:bound}, 
the steepest descent algorithm for $\omega$ with initial point $x \in B^n$
iterates $d({\rm opt} (\bar \omega), x)$ steps 
to obtain an optimal solution $x^*$ of $\bar \omega$.
Hence the total time is $O( d({\rm opt} (\bar \omega), x){\rm MF}(k n, k m))$.

(2).
Let $u:= (x^*)_{\rightarrow y}$.
It suffices to show
\begin{equation*}
f_{ij}(d(u_i, z_j)) \leq 2 \bar f_{ij}(x^*_i, z_j), \quad g_{ij}(d(u_i, u_j)) \leq 2 \bar g_{ij}(x^*_i, x^*_j)
\end{equation*}
for each $i,j$.

We may suppose that $x^*_i  \neq u_i$.
Namely $x^*_i$ is a white vertex, and $d(x^*_i, z_j)$ is odd.
Thus $\bar f_{ij} (d(x^*_i, z_j)) = \{f_{ij} (d(x^*_i, z_j) - 1) + f_{ij} (d(x^*_i, z_j) +1)\}/2$.
If $d(u_i, z_j) = d(x^*_i, z_j) - 1$, then by monotonicity and nonnegativity
we have $\bar f_{ij}(d(u_i, z_j)) \leq \bar f_{ij} (d(x^*_i, z_j)) \leq 2 \bar f_{ij} (d(x^*_i, z_j))$.
If $d(u_i, z_j) = d(x^*_i, z_j) + 1$, then by nonnegativity we have
\[
f_{ij}(d(u_i, z_j)) = 2 \bar f_{ij} (d(x^*_i, z_j)) - f_{ij} (d(x^*_i, z_j) - 1) \leq 2 \bar f_{ij} (d(x^*_i, z_j)).
\] 

Next consider $\bar g_{ij}(d(x^{*}_i, x^{*}_j))$.
Suppose that $x^*_j$ is black. 
Then $x^*_j = u_j$ and $d(x^{*}_i, x^{*}_j)$ is odd.
Thus $\bar g_{ij}(d(x^{*}_i, x^{*}_j)) = \{ g_{ij}(d(x^{*}_i, x^{*}_j) - 1) +  g_{ij}(d(x^{*}_i, x^{*}_j) + 1)\}/2$.
By the same argument above, we have $g_{ij}(d(u_i,u_j)) \leq 2 \bar g_{ij}(d(x^{*}_i, x^{*}_j))$.
Suppose that $x^*_j$ is also white; $d(x^{*}_i, x^{*}_j)$ is even.
If the unique path between $x^*_i$ and $y$ contains $x^*_j$, 
then $d(x^*_i,x^*_j) = d(u_i, u_j)$, and $g_{ij}(d(u_i,u_j)) 
= \bar g_{ij}(d(x^{*}_i, x^{*}_j)) \leq 2 \bar g_{ij}(d(x^{*}_i, x^{*}_j))$.
The same holds for the case where the unique path between $x^*_j$ and $y$ contains $x^*_i$.
Suppose not.
There is a vertex $m$ in the path between $x^*_i$ and $x^*_j$
such that $m$ belongs to the path between $y$ and $x^*_i$ 
and the path between $y$ and $x^*_j$. 
Therefore the rounded $u_i$ and $u_j$ are closer to $m$.
Hence $d(u_i,u_j) \leq d(x^{*}_i, x^{*}_j)$, and 
by monotonicity we have 
$g_{ij}(d(u_i,u_j)) \leq \bar g_{ij}(d(x^{*}_i, x^{*}_j)) \leq 2 \bar g_{ij}(d(x^{*}_i, x^{*}_j))$.

\section{Minimum cost multiflow}\label{sec:multiflow}
In this section, as an application of the results in the previous section, 
we provide
a new simple combinatorial (weakly) polynomial time algorithm to solve 
minimum cost maximum free multiflow problem.
The design of such an algorithm was the original motivation of 
developing the theory of L-extendable functions.

Let ${\cal N}$ be an undirected network on node set $V$, 
edge set $E$, terminal set $S \subseteq V$, and edge-capacity $c: E \to \ZZ_+$.
An {\em $S$-path} in ${\cal N}$ is a path connecting distinct terminals in $S$.
A {\em multiflow} $f$ is a pair $({\cal P}, \lambda)$ of a (multi-)set ${\cal P}$ of $S$-paths and
a nonnegative-valued function $\lambda: {\cal P} \to \RR_+$ satisfying
the capacity constraint:
\begin{equation}
f(e) := \sum \{ \lambda (P) \mid \mbox{$P \in {\cal P}$: $P$ contains $e$ } \}  \leq c(e) \quad (e \in E).
\end{equation}
If $2 \lambda$ is integer-valued, then $f$ is called {\em half-integral}.
For distinct terminals $s,t \in S$,
 let $f(s,t)$ 
denote the total value of the $(s,t)$-flow in $f$, i.e.,
\[
f(s,t) := \sum \{ \lambda (P) \mid \mbox{$P \in {\cal P}$: $P$ is an $(s,t)$-path }\}.
\]
Let $f(s)$ denote the total value on flows connecting $s$, 
i.e., $f(s) := \sum_{t \in S \setminus \{s\}} f(s,t)$.
Then the total flow-value $v_f$ is defined by
\[
v_f := \sum_{P \in {\cal P}} \lambda(P) = \sum_{s,t \in S: s \neq t} f(s,t) = \frac{1}{2} \sum_{s \in S} f(s).
\]
The {\em maximum free multiflow problem} (MF) is:
\begin{description}
\item[(MF)] Find a multiflow $f$ having the maximum total flow value $v_f$. 
\end{description}
A {\em maximum free multiflow} is a multiflow having the maximum total flow-value.

Observe that $f(s)$ is the total flow value of the $(s, S \setminus \{s\})$-flow in $f$, and hence is at most  the minimum value $\kappa_s$ of an $(s, S \setminus \{s\})$-cut.
Thus the total flow-value of any multiflow is at most $\sum_{s \in S} \kappa_s/2$.
A classical theorem by Lov\'asz~\cite{Lov76} and 
Cherkassky~\cite{Cher77} says that 
this bound is always attained by a half-integral multiflow.
\begin{Thm}[\cite{Cher77, Lov76}]\label{thm:LC}
The maximum flow-value of a free multiflow is equal to
\[
\frac{1}{2} \sum_{s \in S} \kappa_s,
\]
and there exists a half-integral maximum free multiflow.
\end{Thm}

Karzanov~\cite{Kar79} considered a minimum cost version of (MF).
Now the network ${\cal N}$ has a nonnegative edge-cost $a: E \to \ZZ_+$.
The total {\em cost} $a_f$ of multiflow $f$ is defined by
$a_f := \sum_{e \in E} a(e) f(e).$
%
%
The {\em minimum cost maximum free multiflow problem} is:
\begin{description}
\item[(MCMF)] Find a maximum free multiflow having the minimum total cost.
\end{description}

There are two approaches to solve this problem. 
The first one is based on the following 
auxiliary maximum multiflow problem with a positive parameter $\mu > 0$:
\begin{description}
\item[(M)] Find a multiflow $f$ maximizing $\mu v_f - a_f$.
\end{description}
If $\mu$ is sufficiently large, then 
an optimal multiflow in (M) is a minimum cost maximum free multiflow. 
Karzanov~\cite{Kar79} showed the half-integrality of (M).
\begin{Thm}[\cite{Kar79}]\label{thm:halfintegral}
For any $\mu \geq 0$, there exists a half-integral optimal multiflow in~$(M)$.
\end{Thm}
In particular there always exists a half-integral minimum cost free multiflow.
The previous algorithms~\cite{GK97, Kar79, Kar94} are based on this formulation.

The second approach, which we will mainly deal with,  
is based on a node-demand multiflow formulation. 
We are further given a nonnegative {\em demand} $r: S \to  \ZZ_+$ 
on terminal set $S$.
A multiflow $f$ is said to be {\em feasible} (to $r$) if $f(s) \geq r(s)$ for $s \in S$.
The {\em minimum cost feasible free multiflow problem} (N) is:
\begin{description}
\item[(N)] Find a feasible multiflow having the minimum total cost.
\end{description}
This problem (N) can solve (MCMF). Indeed, for each $s \in S$, 
add new non-terminal node $\bar s$ 
and new edge $s\bar s$ with capacity $\kappa_s$ and zero cost, and
replace each edge $is$ by $i \bar s$.
Let $r(s) := \kappa_s$ for $s \in S$.
Any feasible multiflow $f$ for the new network must satisfy $f(s) = r(s) = \kappa_s$.
After contracting edges $s\bar s$, 
the resulting $f$ is necessarily a maximum free multiflow in the original network.
Also all maximum free multiflows are obtained in this way.
Indeed, by Theorem~\ref{thm:LC}, 
a maximum free multiflow is simultaneously 
a maximum single commodity $(s, S \setminus \{s\})$-flow for $s \in S$. 
Hence $f(s) = \kappa_s$ must hold, 
and $f$ is extended to a feasible multiflow in the new network.

The problem (N) itself seems natural and fundamental,  but has not been well-studied so far. 
Also we do not know whether (N) reduces to (M), and whether (M) reduces to (N).
Recently Fukunaga~\cite{Fukunaga14} addressed the problem (N) 
in connection with a class of network design problems, 
called (generalized) {\em terminal backup problems}~\cite{AK11, BKM13, XAC08INFOCOM}.
As was noted by him, 
the problem (N) is also formulated as the following cut-covering linear program.
\begin{eqnarray*}
\mbox{\bf (L)} \quad {\rm Min.} &&  \sum_{e \in E} a(e) x(e) \nonumber \\
{\rm s. t. } &&   x(\delta X) \geq r(s) \quad (s \in S, X \in {\cal C}_s), \nonumber \\
&& 0 \leq x(e) \leq c(e) \quad (e \in E),
\end{eqnarray*}
where ${\cal C}_s$ denotes the set of node subsets $X \subseteq V$ 
such that $X$ contains $s$ and does not contain other terminals.
The problems (N) and (L) are equivalent in the following sense.
\begin{Lem} [{see \cite{Fukunaga14}}] \label{lem:NvsL}
\begin{itemize}
\item[{\rm (1)}] For an optimal solution $f$ of (N), 
the flow-support $x: E \to \RR_+$ defined by $x(e) := f(e)$ $(e \in E)$ is an optimal solution of (L).
\item[{\rm (2)}] For an optimal solution $x: E \to \RR_+$ of (L), 
a feasible multiflow $f$ in ${\cal N}$ with the capacity $x$ exists, and
is optimal to (N).
\end{itemize}
In particular the optimal values of the two problems are the same.
\end{Lem}
\begin{proof}
A feasible multiflow $f$ contains an $(s, S \setminus \{s\})$-flow 
with the total flow-value at least $r(s)$. 
The capacity of any $(s, S \setminus \{s\})$-cut under capacity $x$ 
is at least $r(s)$ for $s \in S$.
This means that $x$ is feasible to (L). Hence the optimal value of (L) is at least that of (N).
Conversely, for a feasible solution $x$ of (L), 
consider a maximum free multiflow $f$ under capacity $x$.
By Theorem~\ref{thm:LC}, $f(s)$ is equal 
to the minimum capacity of an $(s, S \setminus \{s\})$-cut under capacity $x$, 
which is at least $r(s)$.
Thus $f$ is feasible, and $a_f \leq \sum_{e \in E} a(e) x(e)$.
This means that the optimal value of (N) is at least that of (L).
\end{proof}
Notice that a feasible multiflow exists (or (L) is feasible) if and only if
\begin{equation}\label{eqn:feasible}
c(\delta X) \geq r(s) \quad (s \in S, X \in {\cal C}_s).
\end{equation}
We will assume this condition in the sequel. Fukunaga~\cite{Fukunaga14} proved 
the half-integrality for (N) and (L).
\begin{Thm}[\cite{Fukunaga14}]\label{thm:fuku}
There exist half-integral optimal solutions in (N) and in (L). They can be 
obtained in strongly polynomial time.
\end{Thm}
The polynomial time solvability depends on a generic LP-solver for solving (L); 
observe that the separation of the feasible region of (L) is done by minimum cut computations, 
and thus (L) is solved by the ellipsoid method.
Also (L) has an extended formulation of polynomial size
\footnote{Instead of exponentially many conditions 
	$x(\delta X) \geq r(s)$, consider 
	a single commodity $(s, S\setminus \{s\})$-flow $\varphi_s$ under capacity $x$ 
	with total flow value at least $r(s)$.
	},  
and thus is solved by the interior point method.

As an application of results in the previous section,  
we present a purely combinatorial polynomial time scaling algorithm 
to obtain half-integral optimal solutions in (N), in (L), and in (MCMF).
The main result in this section is as follows:
\begin{Thm}\label{thm:multiflow_main}
There exists an $O(n \log (n AC)\, {\rm MF}(kn, km))$ time algorithm 
to solve (N), (L), and (MCMF), 
where $n$ is the number of nodes, $m$ is the number of edges,
$k$ is the number of terminals, $A$ is the maximum of edge-costs, 
$C$ is the total sum of edge-capacities.
\end{Thm}
To the best of our knowledge, 
our algorithm is the first combinatorial polynomial time algorithm for (N) and (L), 
and the first combinatorial algorithm for (MCMF) with an explicit polynomial running time.

In Section~\ref{subsec:multiflow_duality},
we  formulate a dual of (N) as a convex location problem 
on a (topological) tree ${\cal T}$ that is the union of the coordinate axises in $\RR^S_+$, 
and establish the half-integrality (Proposition~\ref{prop:min-max}).
Then we  give an optimality criterion (Lemma~\ref{lem:optimality}) for (N).
In Section~\ref{subsec:double_covering}, 
we explain an algorithm to construct an optimal multiflow in (N)
from a given optimal dual solution. This algorithm
is a slight modification of the algorithm of \cite{Kar94} devised for (M). 
In Section~\ref{subsec:proximity_scaling}, 
we present an algorithm to solving a dual of (N),  providing the proof of
Theorem~\ref{thm:multiflow_main}.
By the half-integrality, 
the dual of (N) is the minimization of a 2-separable L-convex function on 
a tree obtained by joining half-integral points in ${\cal T}$.
We will design a proximity scaling algorithm 
by considering a $2$-separable L-convex function minimization
over the tree of $2^{\sigma}$-integral points in each scaling phase $\sigma$.
The time complexity will be estimated by the results of the previous section.
We also sketch how to adapt our algorithm to solve~(M).
In Section~\ref{subsec:additional}, we give additional results and remarks.
In particular, we explain that our combinatorial  
algorithm gives a practical implementation of Fukunaga's 
$4/3$-approximation algorithm~\cite{Fukunaga14} 
for capacitated terminal backup problem.
We also give a further simple and instructive but pseudo-polynomial 
time algorithm to solve (N).

\subsection{Duality}\label{subsec:multiflow_duality}
We first formulate a dual of (N) as a continuous 
location problem on a tree (in topological sense).
Let $\RR^S$ denote the set of functions on $S$.
For each terminal $s \in S$, let $e_s$ denote the function 
defined by $e_s(s) := 1$ and $e_{s}(t) := 0$ for $t \neq s$.
Namely $e_s$ is the $s$th unit vector of $\RR^S$.
Let ${\cal T}_s := \RR_+ e_s$, 
and let ${\cal T} := \bigcup_{s \in S} {\cal T}_s \subseteq \RR^S$.
The metric $D$ on ${\cal T}$ is defined by
\begin{equation*}
D(p,q) := \left\{
\begin{array}{ll}
|p(s) - q(s)| & {\rm if}\  \mbox{$p,q \in {\cal T}_s$ for $s \in S$}, \\
|p(s)| + |q(t)| & {\rm if}\ \mbox{$p \in {\cal T}_s, q \in {\cal T}_t$ for distinct $s,t \in S$}, 
\end{array} \right. \quad (p,q \in {\cal T}).
\end{equation*}
The space ${\cal T}$ is 
isometric to a {\em star} obtained by gluing half-lines $\RR_+$ along the origin.
Notice that $D$ is not equal to an induced metric on $\RR^S$.

Let $V = \{1,2,\ldots,n\}$.
Consider the following continuous location problem on ${\cal T}$:
\begin{eqnarray*}
\mbox{{\bf (D)}: Max.} && \sum_{s \in S} r(s) D(0, p_s) - \sum_{ij \in E} c(ij) (D(p_i, p_j) - a(ij))^+ \\
\mbox{s.t.} && p = (p_1,p_2,\ldots, p_n) \in 
{\cal T} \times {\cal T} \times \cdots \times {\cal T}, \\
                  && p_s \in {\cal T}_s \quad (s \in S), 
\end{eqnarray*}
where  $(z)^+$ denotes $\max (0, z)$.
A feasible solution $p$ of (D) is called a {\em potential}, and called {\em half-integral}
if each $p_i$ is a half-integral vector in $\RR^S$.
A potential $p$ is called {\em proper} if $D(0, p_i) \leq D(0,p_s)$ 
for each $s \in S$ and $i \in V$ with $p_i \in {\cal T}_s$.
\begin{Prop}\label{prop:min-max}
The minimum value of (N) is equal to the maximum value of (D).
Moreover there exists a proper half-integral optimal potential in (D).
\end{Prop}
We will give an algorithmic proof  in Section~\ref{subsub:simple}, 
and here give a sketch of the proof.
\begin{proof}[Sketch of proof]
For a (half-integral) potential $p$ and a terminal $s \in S$, 
let $p'$ be a (half-integral) potential defined 
by $p'_i := p_s$ if $p_i \in {\cal T}_s$ and $D(0, p_i) > D(0, p_s)$, 
and $p'_i := p_i$ otherwise.
Then the objective value of (D) does not decrease.
Therefore there always exists a proper optimal potential in (D).

Let ${\cal C} := \bigcup_{s \in S} {\cal C}_s$.
The LP-dual of (L) is equivalent to:
\begin{eqnarray}\label{eqn:LP-dual'}
\mbox{Max.} &&  \sum_{s \in S} r(s)\sum_{X \in {\cal C}_s} \pi (X) 
- \sum_{e \in E} c(e) \left(\sum_{X \in {\cal C}: e \in \delta X} \pi (X)  - a(e) \right)^+\\
\mbox{s.t.} && \pi: {\cal C}\to \RR_+. \nonumber
\end{eqnarray}
By the standard uncrossing argument, one can show that there always exists an optimal solution $\pi$ 
such that
for $X,Y \in {\rm supp}\, \pi := \{ X \in {\cal C} \mid \pi(X) > 0 \}$, 
it holds $X \subseteq Y$ or $X \supseteq Y$ if $X,Y \in {\cal C}_{s}$ for $s \in S$,
and $X \cap Y = \emptyset$ if $X \in {\cal C}_s$ 
and $Y \in {\cal C}_{s'}$ for distinct $s, s' \in S$. Such a solution is called {\em laminar}.

Thus it suffices to show that 
for a proper potential $p$ there is $\pi: {\cal C} \to \RR_+$ satisfying
\begin{eqnarray}\label{eqn:suffice}
\sum_{X \in {\cal C}: e \in \delta X} \pi(X)   & = & D(p_i, p_j) \quad (ij \in E), \\
\sum_{X \in {\cal C}_s} \pi(X) & = & D(0, p_s) \quad (s \in S), \nonumber
\end{eqnarray}
and that for a laminar solution $\pi: {\cal C} \to \RR_+$,
there is a proper potential $p$ satisfying (\ref{eqn:suffice}).

Let $p = (p_1,p_2,\ldots,p_n)$ be a proper potential.
For $s \in S$ with $p_s \neq 0$, suppose that
$\{p_1,p_2,\ldots,p_n\} \cap ({\cal T}_s \setminus \{0\}) = \{q_1,q_2,\ldots, q_{k_s} = p_s\}$ 
with $0 < D(0, q_1) < D(0, q_{2}) < \cdots < D(0, q_{k_s})$. 
For $j=1,2,\ldots,k_s$, define $X^s_j$ and $\pi^s_j$ by 
\[
X^s_j  :=  \{ i \in V \mid p_i \in \{q_j, q_{j+1},\ldots,q_{k_s}\}\}, \ 
\pi^s_j  := D(q_{j-1}, q_{j}), 
\]
where we let $q_0 := 0$. 
Then $X^s_j \in {\cal C}_s$.
Define $\pi: {\cal C} \to \RR_+$ by $\pi(X^s_j) := \pi^s_j$ and 
$\pi(X) := 0$ for other $X$.
Then (\ref{eqn:suffice}) holds. 

Conversely, let $\pi$ be a laminar solution of (\ref{eqn:LP-dual'}).
Then we can assume that 
${\rm supp}\, \pi \cap {\cal C}_s = \{ X^s_1,X^s_2,\ldots X^s_{k_s}\}$ 
with $X^s_1 \supset X^s_2 \supset \cdots \supset X^s_{k_s} \ni s$.
For each node $i$, if $i$ does not belong to any member of ${\rm supp}\, \pi$,
then define $p_i := 0$.
Otherwise there uniquely exist $s \in S$ and $j \in \{1,2,\ldots,k_s\}$ such that
$i \in X_j^s$ and $i \not \in X_{j+1}^s$, where $X_{k_s+1}^s := \emptyset$.
Define $p_i :=  (\sum_{l= 1}^j \pi(X_l^s)) e_s \in {\cal T}_s$.
Then we obtain a proper potential 
$p = (p_1,p_2,\ldots,p_n)$ of (D) satisfying (\ref{eqn:suffice}).

By Theorem~\ref{thm:fuku}, for every cost vector $a$ (not necessarily nonnegative)
there exists a half-integral optimal solution in (L).
By the total dual (half-)integrality, 
there exists a half-integral laminar optimal solution in (\ref{eqn:LP-dual'}), and
there exists a half-integral optimal potential in (D). 
\end{proof}
We next provide an optimality criterion for (N) and (D).
For a potential $p$, 
an $(s,t)$-path $P = (s = i_0,i_1,\ldots,i_l = t)$ is said to be {\em $p$-geodesic} if 
\[
\sum_{k = 0}^{l-1} D(p_{i_k}, p_{i_{k+1}}) = D(p_s, p_t).
\]
Observe from the triangle inequality that $(\geq)$ always holds.
\begin{Lem}\label{lem:optimality}
A feasible flow $f = ({\cal P}, \lambda)$ 
and a potential $p$ are both optimal if and only if they satisfy the following conditions:
\begin{itemize}
\item[{\rm (1)}] For each edge $ij$, if $D(p_i, p_j) > a(ij)$, then $f(ij) = c(ij)$.
\item[{\rm (2)}] For each edge $ij$, if $D(p_i, p_j) < a(ij)$, then $f(ij) = 0$.
\item[{\rm (3)}] For each path $P$ in ${\cal P}$, if $\lambda(P) > 0$, 
then $P$ is $p$-geodesic.
\item[{\rm (4)}] For each terminal $s$, if $D(0, p_s) > 0$, then $f(s) = r(s)$.
\end{itemize}
\end{Lem}
\begin{proof}
For a path $P = (i_0,i_1,\ldots,i_l)$, 
let $D(p(P)) := \sum_{k = 0}^{l-1} D(p_{i_k}, p_{i_{k+1}})$. 
The statement follows from the previous proposition, and
\begin{eqnarray*}
&& \sum_{ij \in E} a(ij)f(ij) - \sum_{s \in S} r(s) D(0, p_s) + \sum_{ij \in E} c(ij) (D(p_i, p_j) - a(ij))^+ \\
&  & =  \sum_{ij \in E} a(ij)f(ij) + \sum_{ij \in E} c(ij) (D(p_i, p_j) - a(ij))^+ - \sum_{ij \in E} f(ij) D(p_i, p_j) \\
&& \quad + \sum_{ij \in E} f(ij) D(p_i, p_j)  - \sum_{st} f(s,t) D(p_{s}, p_{t})  + \sum_{st} f(s,t) D(p_{s}, p_{t}) -  \sum_{s \in S} r(s) D(0, p_s) \\
&& = \sum_{ij \in E} (D(p_i, p_j) - a(ij))^+ \{ c(ij) - f(ij) \} + 
\sum_{ij \in E} (a(ij) - D(p_i, p_j))^+ f(ij) \\
&& \quad + \sum_{st}\sum_{P \in {\cal P}: \scriptsize\mbox{$P$ connects $s,t$}} \lambda(P) \{ D(p(P))  - D(p_{s}, p_{t}) \}  + \sum_{s \in S} (f(s) - r(s)) D(0, p_s),
\end{eqnarray*} 
where $st$ is taken over all unordered pairs of distinct terminals,  and we use
\begin{eqnarray*}
&& \sum_{ij \in E} f(ij) D(p_i, p_j) = \sum_{ij \in E} \sum_{P \in {\cal P}:e \in P} \lambda(P) D(p_i, p_j) = \sum_{P \in {\cal P}} \lambda(P) D(p(P)), \\
&&
\sum_{st} f(s,t)D(p_{s}, p_{t})  = 
\sum_{st} f(s,t) \left\{ D(p_{s},0) + D(0, p_{t}) \right\} =  \sum_{s \in S} f(s) D(0, p_s).
\end{eqnarray*}
\end{proof}

\subsection{Double covering network}\label{subsec:double_covering}

Here we describe an algorithm to construct 
an optimal multiflow in (N) from an optimal potential $p$ in (D) 
under the condition that each edge-cost is positive: 
\begin{description}
\item[(CP)] $a(e) > 0$ for each edge $e \in E$.
\end{description}
We will see in Remark~\ref{rem:howto} that 
we can assume (CP) by a perturbation technique. 
As Karzanov~\cite{Kar94} did for (M), 
a half-integral optimal multiflow $f$ in (N) is also obtained by 
an integral circulation of a certain directed network ({\em double covering network})
${\cal D}_p$ associated with an optimal potential $p$. 
\begin{figure} 
\begin{center} 
\includegraphics[scale=0.8]{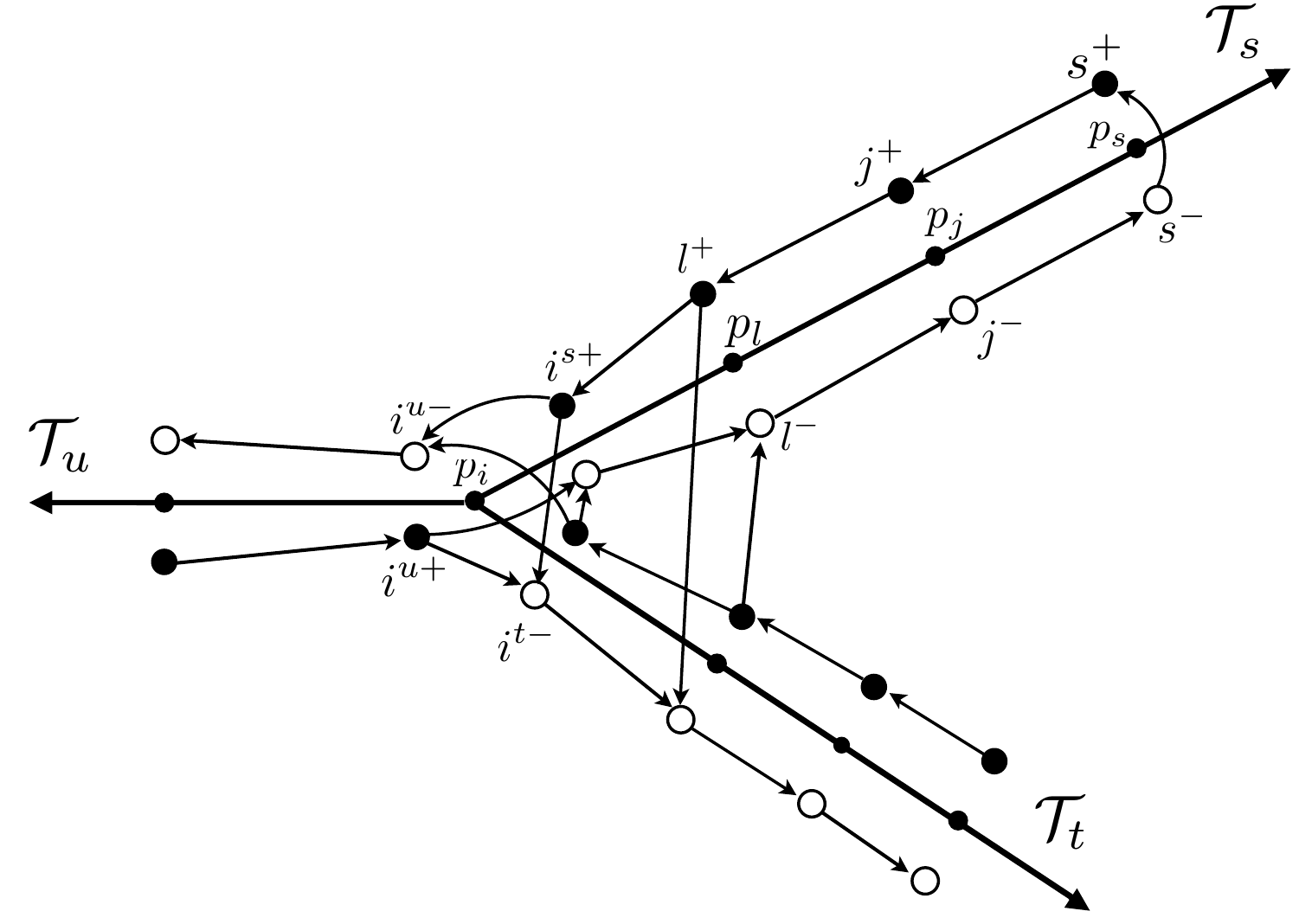}
\caption{Double covering network}  
\label{fig:double_covering}         
\end{center}
\end{figure}

Let $p$ be a (proper) potential.
Let $U_0$ denote the set of non-terminal nodes $i$ with $p_i = 0$.
For each terminal $s \in S$, 
let $U_s$ denote the set of nodes consisting 
of terminal $s$ and non-terminal nodes $i$ with $p_i \in {\cal T}_s \setminus \{0\}$.
Then $V$ is the disjoint union of $U_0$ and $U_s$ for $s \in S$.
Let $E_{=}$ denote the set of edges $ij$ with $D(p_i,p_j) = a(ij)$, 
and let $E_{>}$ denote the set of edges $ij$ with $D(p_i,p_j) > a(ij)$.

The {\em double covering network} ${\cal D}_p$ relative to $p$ is 
a directed network constructed as follows.
For each terminal $s$, consider two nodes $s^+, s^-$. 
For each non-terminal node $i$ not in $U_0$, consider two nodes $i^+, i^-$.
For each (non-terminal) node $i$ in $U_0$, 
consider $2|S|$ nodes
$i^{s+}, i^{s-}$ $(s \in S)$.
The node set of ${\cal D}_p$ consists of these nodes.
Next we define the edge set $A$ of ${\cal D}_p$.
For each edge $ij \in E_{=} \cup E_{>}$, 
define the edge set $A_{ij}$ by:
\begin{equation*}
A_{ij} := \left\{ \begin{array}{ll}
\{ j^+i^+, i^-j^-  \}  & {\rm if}\ i,j \in U_s, D(0,p_i) < D(0,p_j), \\
\{j^+i^{s+}, i^{s-}j^-\} & {\rm if}\ i \in U_0, j  \in U_s, \\
\{ i^+j^-, j^+i^- \}  & {\rm if}\ i \in U_s, j \in U_t, s \neq t.
\end{array}
\right.
\end{equation*}
Notice that for $ij \in E_{=} \cup E_{>}$, 
potentials $p_i$ and $p_j$ are different points in ${\cal T}$ since $a(ij)$ is positive.  
The upper capacity of the two edges in $A_{ij}$ is defined as $c(ij)$.
The lower capacity is defined as 
$0$ if $ij \in E_{=}$ and $c(ij)$ if $ij \in E_{>}$.
For each (non-terminal) node $i$ in $U_0$, 
the edge set $B_i$ is defined as $\{ i^{s+} i^{t-} \mid s,t \in S, s \neq t\}$. 
The lower capacity and the upper capacity of these edges are defined as $0$ and $\infty$, respectively.
For terminal $s \in S$, add edge $s^- s^+$.
The  lower capacity is defined as $r(s)$, 
and the upper capacity is defined as $\infty$ if $p_s = 0$ and $r(s)$ otherwise.
The edge set of ${\cal D}_{p}$ is the (disjoint) union of all edge sets 
$A_{ij}$ $(ij \in E_{=} \cup E_{>})$, 
$B_i$ $(i \in U_0)$, $\{ s^-s^+\}$ $(s \in S)$ (as a multiset). 
As in Figure~\ref{fig:double_covering}, 
readers may imagine that ${\cal D}_p$
is embedded into ${\cal T}$ by the map $i^{\pm} \mapsto p_i$,

Consider an integral feasible circulation 
$\phi: A \to \ZZ_+$ of this network (if it exists).
Decompose $\phi$ into the sum of characteristic vectors of 
directed cycles $C_1,C_2,\ldots, C_{m'}$ with 
positive integral coefficients $q_1,q_2,\ldots,q_{m'}$, 
where $m'$ is at most the number of edges of ${\cal D}_p$.
By construction of ${\cal D}_p$, 
any directed cycle must meet $s^-s^+$ for some terminal $s$, 
and next meets $t^-t^+$ for other terminal $t$ after meeting $s^-s^+$. 
Delete all terminal edges $s^-s^+$ from each $C_l$, 
and obtain directed paths $P_{l}^{1},P_{l}^{2},\ldots, P_l^{n_l}$ $(n_l \leq |S|)$.
Then each $P_l^j$ is a path from $s^+$ to $t^-$ for distinct $s,t \in S$.
Let $\bar P_l^{r}$ be the $S$-path in the original network 
${\cal N}$ obtained from $P_l^{r}$ by replacing $i^{\pm}$ or $i^{s \pm}$ by $i$ 
(and removing $i^{s+} i^{t -}$).
Let ${\cal P}$ be the union of $S$-paths $\bar P_l^{r}$ 
over $l =1,2,\ldots,m', r = 1,2,\ldots, n_l$.
Let $\lambda(\bar P_l^{r}) := q^r_l/2$.
Then $f_{\phi} := ({\cal P},\lambda)$ is a half-integral multiflow; 
we see in the proof of the next lemma that $f_{\phi}$ indeed
satisfies the capacity constraint.

\begin{Prop}\label{prop:f_phi}
A potential $p$ is optimal if and only if there exists 
a feasible circulation in ${\cal D}_p$.
Moreover, for any (integral) feasible circulation $\phi$, 
the (half-integral) multiflow $f_{\phi}$ is optimal to (P).
\end{Prop}
%
\begin{proof}
(Only if part). 
Let $f = ({\cal P},\lambda)$ be an optimal multiflow. 
Then $f$ satisfies the conditions of Lemma~\ref{lem:optimality}.
Consider an path $P = (s = i_0,i_1,\ldots,i_k = t)$ in ${\cal P}$ with $\lambda(P) > 0$.
By condition~(2) with (CP), 
each edge in $P$ belongs to $E_{=} \cup E_{>}$.
By condition~(4), there are an index $l$ such that
$i_0,i_1,\ldots,i_l \in U_s$ and $i_{l+1}, i_{l+2}, \ldots,i_{k} \in U_t$ with 
$D(0,p_{i_0}) > D(0,p_{i_1}) > \cdots > D(0,p_{i_l}) > 0 < 
D(0,p_{i_{l+1}}) < D(0,p_{i_{l+2}}) < \cdots < D(0,p_{i_k})$, 
or
$i_0,i_1,\ldots,i_{l-1} \in U_s$, $i_l \in U_0$, 
and $i_{l+1}, i_{l+2}, \ldots,i_{k} \in U_t$ with 
$D(0,p_{i_0}) > D(0,p_{i_1}) > \cdots > D(0,p_{i_{l-1}}) > 0 = D(0,p_{i_l}) < 
D(0,p_{i_{l+1}}) < D(0,p_{i_{l+2}}) < \cdots < D(0,p_{i_k})$.
For the former case, the union of $A_{ij}$ over edges $ij$ in $P$ 
forms an $(s^+, t^-)$-path and an $(t^+, s^-)$-path.
For the latter case, the union of $\{ i_l^{s+}i_l^{t-}, 
i_l^{t+}i_l^{s-}\}$ and $A_{ij}$ over edges $ij$ in $P$ 
forms an $(s^+, t^-)$-path and an $(t^+, s^-)$-path.
Hence a feasible circulation $\phi_{f}$ is constructed as follows.
For each terminal $s$, define $\phi_f(s^-s^+) := f(s)$. 
For each edge $ij \in E_{=} \cup E_{>}$, 
define $\phi_f(\vec e) := f(ij)$ for $\vec e \in A_{ij}$. 
For each non-terminal node $i \in U_0$ and distinct $s,t \in S$, 
define $\phi_f(i^{s+}i^{t-})$ as the total flow-value 
of $(s,t)$-flows in $f$ using node $i$.
Then the resulting $\phi_f$ is a feasible circulation in ${\cal D}_p$.

(If part). 
We verify that $p$ and $f_{\phi}$ satisfy the conditions of Lemma~\ref{lem:optimality}.
Since there is no edge in ${\cal D}_p$ coming from $ij \in E$ with $a(ij) - D(p_i,p_j) > 0$,
the multiflow $f_\phi$ does not use edge $ij$ with $a(ij) - D(p_i,p_j) > 0$, 
and hence satisfies the condition (2).
Observe that $f_{\phi}(e) = (\phi(e^+) + \phi(e^{-}))/2 (\leq c(e))$ 
for an edge $e = ij \in E_{=} \cup E_{>}$ with $A_{ij} = \{e^+,e^-\}$. 
From this, if $e \in E_{>}$, then $f_{\phi}(e) = (\phi(e^+) + \phi(e^{-}))/2 = c(e)$, 
proving the condition (1).
For terminal $s$, $\phi(s^-s^+)$ is the sum of $q_j$ over indices $j$ 
such that the cycle $C_j$ contains $s^-s^+$, 
which is equal to the sum of $\lambda(\bar P_l^r)$ over 
$S$-paths $\bar P_l^r$ connecting terminal $s$, i.e., $f_{\phi}(s)$.
Thus $f_{\phi}(s) = \phi(s^-s^+) \geq r(s)$; 
in particular $f_{\phi}$ is feasible to $r$. 
Moreover $f_{\phi}(s) = r(s)$ if $s \in U_s$, proving the condition (4).

Finally consider condition (3) for $P = (s = i_0, i_1,\ldots, i_{l} = t) \in {\cal P}$.
Observe from the construction of ${\cal D}_p$ 
that $p_{i_k} \neq p_{i_{k+1}}$,
and $D(p_{i_{k-1}},p_{i_{k+1}}) = D(p_{i_{k-1}},p_{i_{k}}) + D(p_{i_{k}},p_{i_{k+1}})$.
Since the metric space ${\cal T}$ is a tree, we obtain $D(p (P)) = D(p_s,p_t)$; see the next lemma.
%
\end{proof}

In the last part of the proof, we use the following 
distance property of a tree, which we can easily prove (by an inductive argument).
\begin{Lem}
Let $G$ be a tree (with a positive edge-length), 
and let $x = x_0,x_1,\ldots,x_l = y$ be a sequence of vertices in $G$.
Suppose that
\begin{itemize}
\item[{\rm (1)}] $x_i \neq x_{i+1}$ for $i=0,1,2,\ldots,l-1$, and
\item[{\rm (2)}] $d(x_{i-1}, x_{i+1}) = d(x_{i-1}, x_{i}) + d(x_{i}, x_{i+1})$ 
for $i=1,2,\ldots,l-1$.
\end{itemize}
Then $\sum_{i=0}^{l-1} d(x_{i},x_{i+1}) = d(x,y)$.
\end{Lem}
A simple example $(x,z,z,x)$ shows that the condition (1) is necessary.
\begin{Rem}[Role of cost positivity]
One may wonder why the edge-cost positivity (CP) is needed.
Consider the case where some of edges have zero cost.
There may exist edges $ij \in E_{=}$ with $D(p_i, p_j) = 0$.
Therefore we need to add edges to ${\cal D}_p$ corresponding to those edges.
Even if we manage to construct a set ${\cal P}$ of paths from 
a feasible circulation in a modified network,
consecutive nodes in some path $P$ may have the same potential, 
and does not guarantee that $P$ is $p$-geodesic 
($P$ may connect the same terminal).
\end{Rem}

\begin{Rem}[{How to make edge-cost positive}]\label{rem:howto}
The modification is the same as that given in \cite{GK97,Kar94} used for (M).
Let $Z$ denote the set of edges $e$ with $a(e) = 0$. 
Define a positive edge-cost $a'$ by
\begin{equation}\label{eqn:a'}
a'(e) := \left\{ \begin{array}{ll}
1 & {\rm if}\ e \in Z, \\
(2C(Z) + 1) a(e) & {\rm otherwise},
\end{array}\right. \quad (e \in E).
\end{equation}
Then 
any half-integral optimal solution $x$ in (L) with edge-cost $a'$ 
is also optimal to (L) with edge-cost $a$.
Indeed,
by the half-integrality theorem~(Theorem~\ref{thm:fuku}), 
it suffices to show that for every half-integral solution $y$ in (L) with cost $a$
it holds
\[
a x - a y \leq 0, 
\]
where we simply denote $\sum a(e) x(e)$ by $a x$.
Indeed, we have
\[
 (2C(Z) + 1) a x - (2C(Z) +1) a y  = a' x - a' y - x(Z) + y(Z)  \leq C(Z).
\]
This implies that $ax - ay \leq C(Z)/(2C(Z) + 1) < 1/2$.
Since $ax$ and $ay$ are half-integers, we have $ax - ay \leq 0$, as required.
\end{Rem}

\subsection{Proximity scaling algorithm}\label{subsec:proximity_scaling}
In this section we present an algorithm to prove 
Theorem~\ref{thm:multiflow_main}.
By the arguments in the previous section, 
it suffices to solve (D).
Let $\omega: {\cal T}^n \to \overline{\RR}$ be defined by 
\begin{equation}
\omega(p) := \sum_{ij \in E} c(ij) (D(p_i,p_j) -a(ij))^+  + 
\sum_{s \in S} I_s (p_s) - r(s) D(0,p_s) \quad (p \in {\cal T}^n),
\end{equation}
where $I_s$ denotes the indicator function of ${\cal T}_s$:
\[
I_s (q) := \left\{\begin{array}{ll}
0 & {\rm if}\ q \in {\cal T}_s, \\
\infty & {\rm otherwise},
\end{array} \right. (q \in {\cal T}).
\]
Then (D) is equivalent to the minimization of $\omega$.
The range in which an optimum exists is given as follows, where $A := \max \{ a(e) \mid e \in E\}$.
\begin{Lem}\label{lem:region}
There exists a proper half-integral optimal potential $p$
such that $D(0,p_i) \leq n A$ for $i=1,2,\ldots,n$.
\end{Lem}
\begin{proof}
Take a proper half-integral optimal potential $p$.
Suppose that $D(0, p_s) > n A$ for $s \in S$, 
and that $\{p_i \mid i \in U_s\} = \{q_1,q_2,\ldots, q_l = p_s\}$
with $D(0,q_j) < D(0, q_{j+1})$. Let $q_0 := 0$.
Then $l \leq n$ and $\sum_{j=1}^{l} D(q_{j-1}, q_{j}) = D(0, p_s) > n A$.
Thus there is an index $k (\geq 1)$ with $D(q_{k-1}, q_{k}) > A$.
Let $X := \{ i \in U_s \mid p_i \in \{q_k, q_{k+1},\ldots,q_l\}\}$.
For each $ij \in \delta X$, it holds $D(p_{i},p_{j}) \geq D(q_{k-1}, q_{k}) > A \geq a(ij)$. Hence $\delta X \subseteq E_{>}$.
Let $\alpha := D(q_{k-1}, q_{k}) - A > 0$, which is a half-integer. 
Define proper half-integral potential $p'$ by 
\begin{equation}
p'_i := \left\{ \begin{array}{ll}
p_i - \alpha e_s & {\rm if}\ i \in X (\ni s), \\
p_i & {\rm otherwise},
\end{array} \right. \quad (i \in V).
\end{equation}
Then $D(p'_i, p'_j) = D(p_i, p_j) - \alpha$ if $ij \in \delta X$ with $i \in X$, 
and $D(p'_i, p'_j) = D(p_i,p_j)$ otherwise.
Also $D(0, p'_s) = D(0, p_s) - \alpha$ and $D(0, p'_t) = D(0, p_t)$ for other terminal $t \neq s$.
By feasibility~(\ref{eqn:feasible}), 
we obtain
\begin{eqnarray*}
\omega(p') - \omega(p) & = &\sum_{ij \in \delta X} c(ij) \{ D(p'_i, p'_j) - D(p_i,p_j) \} - r(s) \{ D(0, p'_s) - D(0,p_s) \} \\
& = & - \alpha \{ c(\delta X) - r(s) \} \leq 0. 
\end{eqnarray*}
Thus $p'$ is also optimal. 
Let $p := p'$.
Repeat this procedure to obtain an optimal potential $p$ as required.
\end{proof}
Let $L := \lceil \log n A \rceil$, and
let ${\cal T}'$ be the subset of points $q$ of ${\cal T}$ with $D(0, q) \leq 2^{L}$.
By the above lemma, 
(D) is equivalent to the minimization of $\omega$ over $({\cal T}')^n$.
For $\sigma = -1, 0,1,2,\ldots, L$, 
let $T_\sigma$ denote the tree 
on ${\cal T}' \cap (2^{\sigma} \ZZ^S)$ 
such that vertices
$u,v$ are adjacent if $D(u,v) = 2^{\sigma}$.
In particular, $T_\sigma$ is a (graph-theoretical) tree discretizing ${\cal T}'$.
The graph metric of $T_{\sigma}$ is denoted by $d_{\sigma}$.
Then  it holds
\[
2^{\sigma} d_{\sigma}(u,v) = D(u,v).
\]
The two color classes of $T_\sigma$ are denoted by $B_\sigma$ and $W_\sigma$, and 
suppose $0 \in B_\sigma$.
Then $T_{\sigma-1}$ is naturally identified with the subdivision of $T_\sigma$. Hence
\begin{eqnarray}
T_{\sigma-1} = (T_\sigma)^*, \quad B_{\sigma-1} = T_\sigma.
\end{eqnarray}
For $s \in S$, 
define $f_{s, \sigma}: T_{\sigma} \to \overline{\RR}$ by
\[
f_{s, \sigma}(p) := I_s (p) - r(s) 2^{\sigma} d_\sigma(0,p) \quad (p \in T_{\sigma}).
\]
For each edge $ij \in E$, define $g_{ij, \sigma}: \ZZ \to \RR$ by
\[
g_{ij, \sigma}(z) := c(ij) (2^{\sigma} z - a(ij))^+ \quad (z \in \ZZ).
\]
Let $\omega_{\sigma}: {T_{\sigma}}^{n} \to \overline{\RR}$ 
be the restriction of $\omega$ to ${T_{\sigma}}^n$, which is given by
\[
\omega_{\sigma}(p) =  \sum_{s \in S}  f_{s, \sigma}(p_s) + \sum_{ij \in E} g_{ij, \sigma}(d_\sigma(p_i,p_j)) \quad  (p \in {T_{\sigma}}^{n}).
\]
For each edge $ij$, 
consider even function $\bar g_{ij, \sigma}: \ZZ \to \RR$ 
defined as in Section~\ref{subsec:2-separable}.
Namely let
$\bar g_{ij, \sigma}(z) := (g_{ij, \sigma}(z-1) +  g_{ij, \sigma}(z+1))/2$ if $z$ is odd 
and $\bar g_{ij, \sigma}(z) := g_{ij, \sigma}(z)$ if $z$ is even. 
Define $\bar \omega_{\sigma}: {T_{\sigma}}^{n} \to \overline{\RR}$ by
\[
\bar \omega_{\sigma}(p) =  \sum_{s \in S}  f_{s, \sigma}(p_s) + \sum_{ij \in E} \bar g_{ij, \sigma}(d_\sigma(p_i,p_j)) \quad  (p \in {T_{\sigma}}^{n}).
\]
\begin{Lem}\label{lem:omega}
\begin{itemize}
\item[{\rm (1)}] $\omega_{\sigma}$ and $\bar \omega_{\sigma}$ 
are (2-separable) L-extendable and L-convex on ${T_\sigma}^n$, respectively. 
\item[{\rm (2)}] $\bar \omega_{\sigma}$ 
is an L-convex relaxation of $\omega_{\sigma+1}$.
\item[{\rm (3)}] Any minimizer of $\bar \omega_{-1}$  is optimal to (D).
\end{itemize}
\end{Lem} 
\begin{proof}
(1). 
Observe that $f_{s, \sigma}$ is convex on $T_{\sigma}$.
Obviously $g_{ij, \sigma}$ is convex on $\ZZ$. 
Apply Lemma~\ref{lem:tree_convexity} and Theorem~\ref{thm:2separable} 
to obtain the claim.

(2). Let $p \in {T_{\sigma+1}}^n = {B_{\sigma}}^n$. 
Then $d_{{\sigma}}(p_i,p_j)$ is an even integer, 
and $\bar g_{ij, \sigma}(d_{\sigma}(p_i,p_j)) = g_{ij, \sigma}(d_{\sigma}(p_i,p_j))$.
Hence $\omega_{\sigma}(p) = \omega(p) = \omega_{\sigma+1}(p)$.

(3). We show $\bar \omega_{-1} = \omega_{-1}$.
From the view of the proof of (2), 
it suffices to show that $\bar g_{ij,-1}(z)= g_{ij,-1}(z)$ for any odd integer $z$.
Since $a(ij)$ is an integer, either
$(z-1)/2, (z+1)/2 \leq a(ij)$ or $(z-1)/2, (z+1)/2 \geq a(ij)$ holds.
From this, we see $\bar g_{ij,-1}(z)= g_{ij,-1}(z)$.
Notice that $T_{-1}$ is the set of half-integral potentials.
The claim follows from the half-integrality~(Proposition~\ref{prop:min-max}). 
\end{proof}
Thus our goal is to minimize the L-convex function $\omega_{-1}$.
We are now ready to describe our scaling algorithm to solve (D):
\begin{description}
\item[Proximity scaling algorithm:]
\item[Step 0:] Replace $a$ by $a'$ defined by (\ref{eqn:a'}) if $a$ is not positive.
Let $\sigma := L = \lceil \log n A \rceil$ and $p^{\sigma+1}:= (0,0,\ldots,0) \in {B_{\sigma}}^n$.
\item[Step 1:] Find a minimizer $p^{\sigma} \in {T_\sigma}^{n}$ of $\bar \omega_{\sigma}$
by the steepest descent algorithm 
with initial point $p^{\sigma + 1}$.
\item[Step 2:] If $\sigma = -1$, then $p = p^{-1}$ is an optimal solution of (D), and go to step 3.
Otherwise, let $\sigma \leftarrow \sigma - 1$ and go to step 1.
\item[Step 3:] Construct ${\cal D}_p$, and find an integral feasible circulation $\phi$. 
Then $f_{\phi}$ is a half-integral optimal multiflow in (N) as required. 
\end{description}
The time complexity of step 1 is estimated as follows.
\begin{Lem}
$d_{\sigma} ({\rm opt}(\bar \omega_{\sigma}), p^{\sigma + 1}) \leq 6n  + 4$.
\end{Lem}
\begin{proof}
We show the existence of 
a minimizer $q^*$ of $\bar \omega_{\sigma}$ with $d_{\sigma} (q^*, p^{\sigma + 1}) \leq 6n  + 4$.
First consider the case where $\sigma = L$ or $L-1$. 
In this case, the diameter of $T_{\sigma}$ is $2$ or $4$.
Hence the inequality obviously holds. 
Consider the case $\sigma \leq L - 2$.
Then $p^{\sigma+1}$ is a minimizer of an L-convex relaxation 
$\bar \omega_{\sigma+1}$ of $\omega_{\sigma+2}$ (Lemma~\ref{lem:omega}~(2)).
By the persistency (Theorem~\ref{thm:persistency}), 
there exists a minimizer $q$ of $\omega_{\sigma+2}$ 
(over ${T_{\sigma+2}}^{n}$) 
with $d_{\sigma+1} (p^{\sigma+1}, q) \leq 1$.
Since $q$ is also a minimizer of L-extendable function 
$\omega_{\sigma+1}$ (on ${T_{\sigma+1}}^{n} =(B_{\sigma+1} \cup W_{\sigma+1})^{n}$) 
over ${B_{\sigma+1}}^{n}$, 
by proximity theorem (Theorem~\ref{thm:proximity}), 
there is a minimizer $q'$ of $\omega_{\sigma+1}$ (over ${T_{\sigma+1}}^{n})$ 
with $d_{\sigma+1}(q',q) \leq 2 n$.
Since $\bar \omega_{\sigma}$ is an L-convex relaxation of $\omega_{\sigma+1}$, 
the restriction of $\bar \omega_{\sigma}$ to ${B_{\sigma}}^n = {T_{\sigma+1}}^n$ 
is equal to $\omega_{\sigma+1}$.
Hence $q'$ is a minimizer of $\bar \omega_{\sigma}$ over ${B_{\sigma+1}}^n$.
Since $\bar \omega_{\sigma}$ is also midpoint L-extendable  on 
${T_\sigma}^{n} = (B_{\sigma} \cup W_\sigma)^{n}$,
by the proximity theorem,  
there is a minimizer $q^*$ of $\bar \omega_{\sigma}$ over ${T_{\sigma}}^{n}$ 
with $d_{\sigma}(q^*, q') \leq 2 n$.
Notice $2 d_{\sigma +1} = d_{\sigma}$.
Thus we have
\[
d_{\sigma}(p^{\sigma+1},q^*) \leq d_{\sigma}(p^{\sigma+1}, q) + d_{\sigma}(q, q') + d_{\sigma}(q', q^*) 
\leq 2 + 4 n + 2 n
\]
as required.
\end{proof}
Therefore, by Theorem~\ref{thm:bound}, the number of iterations 
of the steepest descent algorithm is at most $6n + 4$.
Since $\bar \omega_{\sigma}$ is a 2-separable L-convex function consisting of $O(m)$ terms,
and the maximum degree of $T_{\sigma}$ is the number $k$ of terminals, 
by Theorem~\ref{thm:approx} we can find an optimal solution in 
$O( n {\rm MF}(kn, km))$ time. 
The total step is $O( n L \, {\rm MF}(kn,km))$, 
where $L$ can be taken to be $\lceil \log 2 (C(Z) + 1) n A \rceil = O(\log n A C)$.
This proves Theorem~\ref{thm:multiflow_main}.

\paragraph{Algorithm for (M).}
Let us sketch a proximity scaling algorithm to solve (M).
Corresponding to (D), consider the following location problem on ${\cal T}$.
\begin{eqnarray*}
\mbox{(D$'$): Max.} && \sum_{ij \in E} c(ij) (D(p_i, p_j) - a(ij))^+ \\
\mbox{s.t.} && p = (p_1,p_2,\ldots, p_n) \in 
{\cal T}^n, \\
                  && p_s = \mu e_s/2 \quad (s \in S).
\end{eqnarray*}
Again a feasible solution of (D$'$) is called a potential, and called half-integral 
if each $p_i$ is half-integral. 
The following duality is implicit in \cite{ HHMPA, Kar94}.
\begin{Prop}[see \cite{HHMPA, Kar94}]
The maximum value of (M) is equal to the minimum value of (D$'$).
Moreover, if $\mu$ is an integer, then there exists a half-integral optimal potential in (D$'$).
\end{Prop}
\begin{proof}[Sketch of proof]
The edge-capacitated formulation is transformed to 
the node-capacitated formulation, discussed in \cite{HHMPA}, as follows.
Replace each edge $e = ij$ by the series of two edges $iu$ and $uj$.
Define the node-capacity and the node-cost of the new node $u$ by $c(e)$ and $a(e)$.
No edge-capacity and edge-cost are given. The node-capacities of the original nodes are $\infty$.
Then (M) becomes a node-capacitated problem, 
and the results in \cite[Section 4]{HHMPA} are applicable.
In particular, the dual of (M) is given by the problem (4.6) of \cite{HHMPA} 
in setting $\bar \mGamma := {\cal T}$ and $\bar R_s := \{ \mu e_s /2\}$.
Subtree $F(i)$ for the original node $i$ is a single point $p_i$ 
(by $b(x) = \infty$), 
and hence $F(u)$ for new node $u$ replacing original edge $ij$ 
is a path between $p_i$ and $p_j$ with length (diameter) $D(p_i, p_j)$.
Thus (4.6) of \cite{HHMPA} becomes (D$'$). 
The half-integrality follows from \cite[Remark 4.7]{HHMPA}.
\end{proof}
By the same argument in the proof of Lemma~\ref{lem:optimality}, 
one can prove that a multiflow $f$ and a potential $p$ are both optimal 
if and only if they satisfy conditions (1), (2), and (3) in Lemma~\ref{lem:optimality}. 
The corresponding double covering network ${\cal D}'_p$ is obtained by 
replacing the lower bound and the upper bound capacities of $s^-s^+$ of ${\cal D}_p$ 
by $0$ and $\infty$, respectively.
Then we obtain an analogue of Proposition~\ref{prop:f_phi} that 
$p$ is optimal if and only if there exists a feasible circulation $\phi$ in ${\cal D}'_p$, 
and $f_{\phi}$ is an optimal multiflow in (M).
We may consider that the variables of (D$'$) are $p_i$ 
for non-terminal nodes $i  \in V \setminus S$ 
(since a potential $p_s$ of each terminal $s$ is fixed to $\mu e_s /2$ in (D$'$)).  
For non-terminal node $i$, 
define $f_i: {\cal T} \to \RR_+$ by 
$f_i (q) := \sum_{s \in S: si \in E} c(si) D(\mu e_s/2, q)$.
We may assume that the set of non-terminal nodes are $\{1,2,\ldots,n-k\}$.
Then (D$'$) is the minimization of 
$\omega (p) := \sum_{1 \leq i \leq n-k} f_i( p_i) 
+ \sum_{ij \in E: 1 \leq i,j \leq n-k} c(ij) (D(p_i, p_j)- a(ij))^+$.
Again it is easy to see that there is an optimal potential $p$ 
with $D(0,p_i) \leq \mu/2$.
For $\sigma = -1, 0,1,2,\ldots$, define $T_{\sigma}$, $g_{ij,\sigma}$, 
$\bar g_{ij,\sigma}$, $\omega_{\sigma}$, and $\bar \omega_{\sigma}$ as above.
Then Lemma~\ref{lem:omega} holds in this setting.
Let $L := \lceil \log \mu \rceil$.
By the proximity scaling algorithm,
we can minimize $\omega'$ in $O((n-k) \lceil \log \mu \rceil {\rm MF}(kn, km))$ time.

It is shown in \cite{GK97, Kar94} 
that if $\mu \geq 2A_1C + 1$ for $A_1 := \sum_{e \in E} a(e)$ 
and $C = \sum_{e \in E} c(e)$, 
then every half-integral optimal multiflow in (M)
is a minimum cost multiflow.
Also it is shown in \cite{GK97, Kar94} 
that cost $a$ is replaced by $a'$ (defined by (\ref{eqn:a'})) 
to satisfy the cost positivity.
Any half-integral optimal multiflow in (M) with cost $a'$ 
is optimal for original cost $a$.
Thus, letting $\mu := 2 A'_1C + 1 = O(A_1 C)$, 
we obtain a minimum cost half-integral multiflow 
in $O( (n-k) \log (A_1 C ) {\rm MF} (kn, km))$ time. 

\subsection{Additional results and remarks}\label{subsec:additional}
\subsubsection{Lov\'asz-Cherkassky formula, $k$-submodular function, and multiway cut}
\label{subsub:LC}

We note a relation
between Lov\'asz-Cherkassky formula (Theorem~\ref{thm:LC}), 
$k$-submodular function minimization, and multiway cut.
Suppose that $S$ consists of $k$ terminals, and 
the set of non-terminal nodes is $\{1,2,\ldots,n\}$.
Recall notions in Section~\ref{subsec:k-submodular}. 
Adding $0$ to $S$, we obtain poset $S_k$, and 
consider the following $k$-submodular function minimization:
\begin{eqnarray}\label{eqn:LS}
\mbox{Min.}  && \frac{1}{2} \sum_{1 \leq i \leq n} \sum_{s \in S: si \in E} 
c(si) \delta (s, p_i) + \frac{1}{2} \sum_{ ij \in E: 1 \leq i,j \leq n} c(ij) \delta (p_i, p_j) 
 \\
\mbox{s.t.} && p = (p_1,p_2, \ldots,p_n) \in {S_k}^n.\nonumber
\end{eqnarray}
Recall Example~\ref{rem:multiway} with $d = \delta$
that this problem is nothing but 
a $k$-submodular relaxation (or an L-convex relaxation) 
of multiway cut.

Furthermore this problem is also a dual of maximum free multiflow problem (MF),
and hence the optimal value of this problem is equal to $\sum_{s \in S} \kappa_s/2$.
To see this, for $p \in {S_k}^n$ and $s \in S$, let 
$X^p_s := \{ s\} \cup \{i \mid p_i = s\}$.
Then $X^p_s$ is an $(s, S \setminus \{s\})$-cut. 
Observe $\sum_{s \in S} c(\delta X^{p}_s)/2$ 
is equal to the objective value of (\ref{eqn:LS}) at $p$.
Conversely, take a minimum $(s, S \setminus \{s\})$-cut $X_s$ for each $s \in S$.
We can assume that $X_s$ $(s \in S)$ are disjoint. 
If $X_s \cap X_t \neq \emptyset$, then, by submodularity, 
we can replace 
$X_s$ by $X_s \setminus X_t$ and $X_t$ by $X_t \setminus X_s$ 
without increasing the cut capacity.
Define $p$ by $p_i = s$ if $i \in X_s$ for some $s \in S$ and $p_i = 0$ otherwise.  
Then the objective value at $p$ is equal to 
$\sum_{s \in S} c(\delta X_s)/2 = \sum_{s \in S} \kappa_s/2$.
In particular, this $k$-submodular function minimization can be solved 
by $(s, S\setminus\{s\})$-mincut computation for each $s \in S$.

\subsubsection{Application to terminal backup problem}

The linear program (L) arises as an LP-relaxation 
of a class of network design problems, 
called {\em terminal backup problems}~\cite{AK11, XAC08INFOCOM}.
Given a graph $G = (V,E)$ with terminal set $S$ and edge-cost $a: E \to \ZZ_+$ 
the terminal backup problem asks to find
a minimum cost subgraph $F$ with the property 
that each terminal $s$ is reachable to other terminal in $F$.
Anshelevich and Karagiozova~\cite{AK11} proved that this problem is solvable 
in polynomial time.
Bern{\'a}th, Kobayashi, and Matsuoka~\cite{BKM13} considered the following weighted generalization.
Given a graph $G = (V,E)$ with 
terminal set $S$, edge-cost $a: E \to \ZZ_+$, and a requirement $r: S \to \ZZ_+$, 
find a minimum cost integral edge-capacity $x: E \to \ZZ_+$ 
such that for each terminal $s$ there is an integral $(s, S \setminus \{s\})$-flow in $(V,E,x)$
with flow-value at least $r(s)$. 
They proved that this generalization is solvable in (strongly) polynomial time, 
and asked whether a natural capacitated version of this problem 
is tractable or not.
The capacitated version is 
to impose the condition $x(e) \leq c(e)$ $(e \in E)$ 
for $c: E \to \ZZ_+$, and is formulated as 
the following integer program:
\begin{eqnarray*}
\mbox{(CTB)}\quad  {\rm Min.} &&  \sum_{e \in E} a(e) x(e) \nonumber \\
{\rm s. t. } &&  x(\delta X) \geq r(s) \quad (s \in S, X \in {\cal C}_s), \nonumber \\
&& x(e) \in \{0,1,2,\ldots, c(e)\} \quad (e \in E).
\end{eqnarray*}
The problem (L) is nothing but a natural LP-relaxation of (CTB). 

Fukunaga~\cite{Fukunaga14} studied (CTB), and proved the half-integrality 
(Theorem~\ref{thm:fuku}) of the LP-relaxation (L).  
As was noted by him,  
a $2$-approximation solution is immediately obtained 
from a half-integral optimal solution $x$ in (L) by rounding 
each non-integral component $x(e)$ to $x(e) +1/2$.
He devised a clever rounding algorithm to obtain a $4/3$-approximation solution.
Our algorithm gives a practical and combinatorial implementation of 
his $4/3$-approximation algorithm, as follows.
A half-integral optimal solution $x$ of (L) and an optimal potential $p$ of (D) are obtained by our combinatorial algorithm.
Fukunaga's algorithm rounds a special half-integral optimal solution $\tilde x$ 
obtained from $x$ by the following fixing procedure. 
Let $E_1$ be the set of edges $e$ with $x(e) \in \ZZ$.
Let $\tilde x (e):= x(e)$ for $e \in E_1$.
For an edge $e\in E \setminus E_1$ (with non-integral $x(e)$), 
check whether there is an optimal solution $y$ in (L)
such that $y(e) \in \{\lfloor x(e) \rfloor, \lceil x(e) \rceil\}$, 
$\lfloor x(e') \rfloor \leq y(e') \leq  \lceil x(e') \rceil$ for $e' \in E \setminus (E_1 \cup \{e\})$, 
and $y(e') = \tilde x(e')$ for $e' \in E_1$.
If such $y$ exists, then let $\tilde x(e) := y(e)$.
Otherwise let $\tilde x(e) := x(e)$. 
Add $e$ to $E_1$, and repeat until $E_1 = E$ to obtain $\tilde x$.
This procedure can be implemented on 
the double covering network ${\cal D}_p$.
By Lemma~\ref{lem:NvsL} with (CP), 
$y$ is optimal to (L)  if and only if 
$y$ is the flow-support of some optimal multiflow $f$ in (N).
From view of (the proof of) Proposition~\ref{prop:f_phi},
$y$ is optimal if and only if there is a circulation $\phi$ of ${\cal D}_p$ 
with $y(ij) = \phi(\vec e)$ for $ij \in E, \vec e \in A_{ij}$.
Therefore the above procedure 
reduces to checking the existence of a circulation 
in ${\cal D}_p$ with changing the lower and upper capacities
of edges in $A_{ij}$ to $\lfloor x(e) \rfloor$ or $\lceil x(e) \rceil$ appropriately.
Thus $\tilde x$ is obtained 
by at most $m$ max-flow computations on ${\cal D}_p$.

\subsubsection{Simple descent algorithm by double covering network}\label{subsub:simple}
We here present a simple and instructive but pseudo-polynomial time algorithm 
solving (N) and (D) of the following description:
\begin{quote}
For a potential $p$, 
find a feasible circulation $\phi$ in ${\cal D}_p$.
If $\phi$ exists, then $f_{\phi}$ is an optimal multiflow, and stop.
Otherwise, from an infeasibility certificate of ${\cal D}_p$, 
we obtain another potential $p'$ with 
$\omega(p') < \omega(p)$. Let $p \rightarrow p'$ and repeat.
\end{quote}
The presented algorithm can always keep $p$ half-integral, 
providing an algorithmic proof 
of Proposition~\ref{prop:min-max}.

Assume the cost positivity (CP).
For a (proper) half-integral potential $p$, construct 
the double covering network ${\cal D}_p$, as above.
We reduce the circulation problem on ${\cal D}_p$ to the maximum flow problem 
on a modified network  $\tilde {\cal D}_{p}$.
Consider supper source $a^{+}$ and sink $a^{-}$.
For each $ij \in E_{>0}$
modify edge set $A_{ij}$ as follows. 
For each $uv \in A_{ij}$
replace $uv$ by two edges 
$ua^{-}, a^{+} v$ with (upper-)capacity $c(ij)$ (and lower-capacity $0$).
%
For each terminal $s \in S$, 
add new two edges $s^- a^+$ and $s^+ a^{-}$ with capacity $r(s)$.
For edge $s^{-}s^{+}$,  
change the lower-capacity to $0$ and the upper-capacity 
to $\infty$ if $p_s = 0$ and to $0$ otherwise.
The resulting network is denoted by $\tilde {\cal D}_p$.
Consider the maximum $(a^+,a^-)$-flow problem on the new network $\tilde {\cal D}_{p}$.
This is a standard reduction of a circulation problem to a max-flow problem.
In particular,  ${\cal D}_p$ has a feasible circulation if and only if 
a maximum $(a^+, a^-)$-flow $\tilde {\cal D}_p$ 
saturates all edges leaving $a^{+}$ (entering $a^-$), i.e.,
$\{a^+\}$ is a minimum $(a^+,a^-)$-cut.

%
Let  $\tilde V_1$ (resp. $\tilde V_{2}$) be 
the set of nodes $i^+$, $i^-$, $i^{s+}$, or $i^{s-}$
such that $i$ has an integral potential $p_i$ (resp. non-integral potential $p_i$).
By the integrality of $a(ij)$, 
there is no edge $ij$ in $E_{=}$ such that 
$i$ has an integral potential and
$j$ has a non-integral potential.
Hence there is no edge connecting between $\tilde V_{1}$ and $\tilde V_{2}$.

An $(a^+, a^{-})$-cut $X$ in $\tilde {\cal D}_p$ is called {\em legal} 
if 
\begin{itemize}
\item[(1)] $X \cap \tilde V_1$ or $X \cap \tilde V_2$ is empty,
\item[(2)] for each $i \in U_0$, $X \cap B_{i}$ is empty or 
$\{ i^{s+}\} \cup \{i^{t-} \mid t \in S \setminus \{s\}\}$ for some $s \in S$, and
\item[(3)] for other node $i$,   $X \cap \{i^+,i^-\}$ is empty, $\{i^+\}$, or $\{i^-\}$. 
\end{itemize}
For a legal cut $X$ of $\tilde D_{p}$, 
the potential $p^X$ is defined by: 
\begin{equation}
p^X_i := \left\{ \begin{array}{ll}
e_s/2 & {\rm if}\  \mbox{$i \in U_0$ and $i^{s+} \in X$ for $s \in S$},\\
p_i + e_s/2 & {\rm if}\ \mbox{$i^{+} \in X$ and $i \in U_s$ for $s \in S$}, \\ 
p_i - e_s/2 & {\rm if}\ \mbox{$i^{-} \in X$ and $i \in U_s$ for $s \in S$},\\
p_i & {\rm otherwise},
\end{array} \right. \quad (i \in V).
\end{equation}
Then the following lemma holds; the proofs are given in the end of this section.
\begin{Lem}\label{lem:legal}
\begin{itemize}
\item[{\rm (1)}]
For a legal cut $X$ in $\tilde {\cal D}_p$, we have
\[
\omega(p^X) - \omega(p) = \frac{1}{2} \{ c(\delta X) - c(\delta \{a^{+}\}) \}.
\]
\item[{\rm (2)}]
Let $X$ be a (unique) inclusion-minimal minimum $(a^+,a^-)$-cut in $\tilde D_p$, 
and let $X_1 := X \setminus \tilde V_2$ and $X_2:= X \setminus \tilde V_1$.  
Then both $X_1$ and $X_2$ are legal cuts with
\[
c(\delta X) = c(\delta X_1) + c(\delta X_2) - c( \delta \{a^+\}).
\]
In particular, if $c(\delta X) < c(\delta \{a^+\})$, 
then  $c(\delta X_1) < c(\delta \{a^+\})$ or  $c(\delta X_2) < c(\delta \{a^+\})$.
\end{itemize}
\end{Lem}

Therefore we can check the optimality of $p$ by solving 
the maximum-flow problem on $\tilde {\cal D}_p$. 
If $p$ is not optimal, then we obtain another half-integral potential $p^X$ 
having a smaller objective value. This naturally provides the following algorithm:
\begin{description}
\item[Descent algorithm by double covering network]
\item[Step 0:] Replace $a$ by $a'$ defined by (\ref{eqn:a'}) if $a$ is not positive. 
 Let $p := (0,0,\ldots,0)$.
\item[Step 1:] Construct $\tilde D_{p}$, and 
find a minimal minimum $(a^+,a^-)$-cut $X$ and a maximum $(a^+,a^-)$-flow $\tilde \phi$.
\item[Step 2:] If $X = \{a^+\}$, then $p$ is optimal, and
construct a feasible circulation $\phi$ on ${\cal D}_p$ from $\tilde \phi$ and an optimal multiflow $f_{\phi}$ from $\phi$; stop. 
Otherwise go to step 3.
\item[Step 3:] Let $X_1 := X  \setminus \tilde V_2$ and $X_2 := X  \setminus \tilde V_1$.
Choose $j \in \{1,2\}$ with $c(\delta X_j) < c(\delta \{a^+\})$. Let $p := p^{X_j}$ and go to step 1.
\end{description}
Observe that this algorithm coincides with the steepest descent algorithm applied to 
L-convex function $\omega_{-1}$ on ${T_{-1}}^n$, where 
$p^{X_j}$ is a steepest direction of ${\cal I}(p)$ 
for even iterations and of ${\cal F}(p)$ for odd iterations. 
Thus, by Theorem~\ref{thm:bound} and Lemma~\ref{lem:region}, 
the number of the iteration is $O(n AC )$.
The numbers of nodes and edges of $\tilde {\cal D}_p$ are 
$O(kn)$ and $O(m+ k^2n)$, respectively.
Thus we have:
\begin{Thm}
The above algorithm runs in $O(n AC\, {\rm MF}( kn, m + k^2 n ))$ time.
\end{Thm}
\begin{Rem}
In the above algorithm, 
each step minimizes a sum of 
basic ${\pmb k}$-submodular functions of type I and III (thanks to (CP)).
The above network $\tilde {\cal D}_p$ may be viewed as 
yet another representation of ${\pmb k}$-submodular functions.
In fact, an arbitrary sum of basic ${\pmb k}$-submodular functions of type I and III 
admits this kind of a network representation 
(Yuta Ishii, Master Thesis, University of Tokyo, 2014).
However this representation seems not to capture basic ${\pmb k}$-submodular functions of type II. 
In each scaling phase of the proximity scaling algorithm, 
the objective functions of local problems may contain a ${\pmb k}$-submodular term of type II. 
This is why we need an algorithm in Theorem~\ref{thm:IWY14}.
\end{Rem}

\begin{proof}[Proof of Lemma~\ref{lem:legal}]
(1).  Let $X$ be a legal cut (with finite cut capacity).
Observe that
\begin{equation}
c(\delta X) = \sum_{ij \in E_{=} \cup E_{>}} c(ij) | \delta X \cap A_{ij}|
+ \sum_{s \in S} r(s) |\delta X \cap \{ a^+ s^+, s^-a^- \}|.
\end{equation}
In particular we have
\begin{equation}
c(\delta \{a^+\}) = \sum_{ij \in E_{>}} 2 c(ij) + \sum_{s \in S} r(s).
\end{equation}
Let $p' := p^X$. 
For an edge $ij \in E$,
if $p_i$ and $p_j$ are integral and non-integral potentials, respectively, 
then $ij \not \in E_{=}$ and $D(p_i,p_j) -1/2 \leq D(p'_i, p'_j) \leq D(p_i,p_j) +1/2$.
Therefore $a(ij) > D(p_i,p_j)$ implies $a(ij) \geq D(p'_i,p'_j)$ and
$a(ij) < D(p_i,p_j)$ implies $a(ij) \leq  D(p'_i,p'_j)$.
Consequently we have
\begin{eqnarray*}
\omega (p')-  \omega (p) &= &\sum_{ij \in E_{>}} c(ij) (D(p'_i, p'_j)- D(p_i, p_j) ) + \sum_{ij \in E_{=}} c(ij) (D(p'_i, p'_j) - a(ij))^+ \\
&& - \sum_{s \in S} r(s) ( D(0, p'_s) - D(0, p_s) ). \\
& = & \sum_{ij \in E_{>}} c(ij) ( D(p'_i, p'_j)- D(p_i, p_j) + 1) + \sum_{ij \in E_{=}} c(ij) (D(p'_i, p'_j) - a(ij))^+ \\
&& + \sum_{s \in S} r(s) ( D(0, p_s) - D(0, p'_s) + 1/2) -  c( \delta \{a^+\})/2.
\end{eqnarray*}
It suffices to show that
\begin{eqnarray}
 |\delta X \cap A_{ij}|/2 & = & D(p'_i, p'_j)- D(p_i, p_j) + 1 \quad (ij \in E_{>}), \label{eqn:1} \\
 |\delta X \cap A_{ij}|/2  & = & (D(p'_i, p'_j) - a(ij))^+  \quad (ij \in E_=), \label{eqn:2}\\
|\delta X \cap \{ a^+s^+, s^- a^- \}|/2 & = & D(0, p_s) - D(0, p'_s) + 1/2 \quad (s \in S). \label{eqn:3}
\end{eqnarray}
Pick $ij \in E_{=} \cup E_{>}$.

Case 1: $i \in U_s$, $j \in U_s \cup U_0$ and $D(0,p_i) > D(0,p_j)$. 
If $j \in U_0$, then $j^{s\pm}$ is denoted by $j^{\pm}$.
Suppose $ij \in E_{=}$ (to show (\ref{eqn:2})).
If $\{i^+,i^-,j^+,j^-\} \cap X$ is empty or 
contains $i^-$ or $j^+$, 
then $\delta X \cap A_{ij}$ is empty, 
$D(p'_i,p'_j) - D(p_i,p_j) \leq 0$ and hence $(D(p'_i,p'_j) - a(ij))^+ = 0$.
If  $\{i^+,i^-,j^+,j^-\} \cap X = \{i^+\}$ or $\{j^-\}$, 
then $|\delta X \cap A_{ij}| = 1$, 
$D(p'_i,p'_j) - D(p_i,p_j) = 1/2$, and  $(D(p'_i,p'_j) - a(ij))^+ = 1/2$.
If $\{i^+,i^-,j^+,j^-\} \cap X = \{i^+,j^-\}$, 
then $|\delta X \cap A_{ij}| = 2$, 
$D(p'_i,p'_j) - D(p_i,p_j) = 1$, and $(D(p'_i,p'_j) - a(ij))^+ = 1$.
Suppose $ij \in E_{>}$ (to show (\ref{eqn:1})).
If $\{i^+,i^-,j^+,j^-\} \cap X = \{i^+,j^+\}$, $\{i^-,j^-\}$ or empty, 
then  $|\delta X \cap A_{ij}| = 2$, 
$D(p'_i,p'_j) = D(p_i,p_j)$, and $D(p'_i,p'_j) - D(p_i,p_j) +1 = 1$.
If $\{i^+,i^-,j^+,j^-\} \cap X = \{i^+\}$ or $\{j^-\}$, 
then $|\delta X \cap A_{ij}| = 3$, 
$D(p'_i,p'_j) = D(p_i,p_j) + 1/2$, and  $D(p'_i,p'_j) - D(p_i,p_j) +1 = 3/2$.
If $\{i^+,i^-,j^+,j^-\} \cap X = \{i^-\}$ or $\{j^+\}$, 
then $|\delta X \cap A_{ij}| = 1$,
$D(p'_i,p'_j) = D(p_i,p_j) - 1/2$, and  $D(p'_i,p'_j) - D(p_i,p_j) +1 = 1/2$.
If $\{i^+,i^-,j^+,j^-\} \cap X = \{i^+, j^-\}$, 
then $|\delta X \cap A_{ij}| = 4$,
$D(p'_i,p'_j) = D(p_i,p_j) + 1$, and  $D(p'_i,p'_j) - D(p_i,p_j) +1 = 2$.
If $\{i^+,i^-,j^+,j^-\} \cap X = \{i^-, j^+\}$, 
then $|\delta X \cap A_{ij}| = 0$,
$D(p'_i,p'_j) = D(p_i,p_j) - 1$, and  $D(p'_i,p'_j) - D(p_i,p_j) +1 = 0$.

Case 2: $i \in U_s$ and $j \in U_{s'}$.
In this case, (\ref{eqn:1}) and (\ref{eqn:2}) 
are obtained by replacing roles of $j^+$  and 
$j^-$ in the above case 1.

Next consider terminal $s \in S$ (to show (\ref{eqn:3})).
If $X \cap \{s^+,s^-\} = \{s^+\}$, 
then $D(0,p'_s) = D(0,p_s) + 1/2$ and 
$\delta X \cap \{a^+s^+, s^-a^-\}$ is empty.
If $X \cap \{s^+,s^-\} = \{s^-\}$, 
then $D(0,p'_s) = D(0,p_s) - 1/2$ and 
$|\delta X \cap \{a^+s^+, s^-a^-\}| = 2$.
If $X \cap \{s^+,s^-\}$ is empty, 
then $D(0,p'_s) = D(0,p_s)$ and 
$|\delta X \cap \{a^+s^+, s^-a^-\}| = 1$.
For all cases,  (\ref{eqn:3}) holds.

(2). 
The equality $c(\delta X) = c(\delta X_1) + c(\delta X_2) - c( \delta \{a^+\})$
follows from the fact that there is no edge between $\tilde V_1$ and $\tilde V_2$.
So it suffices to show that minimal minimum $(a^+,a^-)$-cut $X$ 
satisfies the conditions (2) and (3) of legal cuts.

Suppose (for contradiction) that $X$ contains 
$\{i^+,i^-\}$ or $\{i^{s+}, i^{s-}\}$.
Remove all such pairs of nodes 
from $X$ to obtain another $(a^+,a^-)$-cut $X'$.
Observe that $|A_{ij} \cap \delta X| \geq |A_{ij} \cap \delta X'|$ for each $ij \in E_{=} \cup E_{>}$, 
and $|\{a^+s^+, s^-a^-\} \cap \delta X| \geq |\{a^+s^+, s^-a^-\} \cap \delta X'|$ 
for each terminal $s$.
Moreover $\delta X'$ does not contain edge $i^{s +}i^{s'-}$ (of infinite capacity).
Otherwise $i^{s+}, i^{s'-}, i^{s'+} \in X \not \ni i^{s-}$, and 
$\delta X$ has edge $i^{s'+}i^{s-}$ of infinite capacity; a contradiction.
Thus $X'$ is also a minimum cut, contradicting the minimality of $X$.
Thus $\delta X$ cannot contain both $i^{+}$ ($i^{s+}$) and $i^{-}$ ($i^{s-}$).
Suppose for contradiction that
$X$ contains $i^{s-}$ and does not contain $i^{s'+}$ for each $s' \in S \setminus \{s\}$.
In this case, remove $i^{s-}$ from $X$. 
Then the cut capacity does not increase, contradicting the minimality of $X$.
Hence $X$ satisfies (2) and (3), as required.
\end{proof}

\section*{Acknowledgments}
The author thanks Takuro Fukunaga for communicating \cite{Fukunaga14}, 
Vladimir Kolmogorov 
for helpful comments on the earlier version of this paper, 
and the referee for helpful comments.
The work was partially supported by JSPS KAKENHI Grant Numbers 25280004, 26330023, 26280004.

\end{document}